\documentclass[11pt]{amsart}
\title[Stable auto-equivalence]{A generalization of Dugas' construction on stable auto-equivalences for symmetric algebras}
\author{Nengqun Li and Yuming Liu*}

\address{Nengqun Li and Yuming Liu
\newline School of Mathematical Sciences
\newline Laboratory of Mathematics and Complex Systems
\newline Beijing Normal University
\newline Beijing 100875
\newline P.R.China}
\email{ymliu@bnu.edu.cn}
\email{wd0843@163.com}

\date{version of \today}
\usepackage{amsmath}
\usepackage{amssymb}
\usepackage{amsxtra}
\usepackage{abstract}
\usepackage{mathrsfs}
\usepackage{mathtools}
\usepackage[all]{xy}
\usepackage{amscd}
\usepackage{amsthm}
\usepackage[dvips]{graphicx}
\usepackage{ulem}
\usepackage{color}
\usepackage{appendix}
\usepackage{tikz}

\newtheorem{Thm}{Theorem}[section]
\newtheorem{Lem}[Thm]{Lemma}
\newtheorem{Def}[Thm]{Definition}
\newtheorem{Cor}[Thm]{Corollary}
\newtheorem{Prop}[Thm]{Proposition}
\newtheorem{Ex1}[Thm]{Example}
\newtheorem{Rem1}[Thm]{Remark}

\newcommand{\lra}{\longrightarrow}

\newcommand{\ra}{\rightarrow}
\newcommand{\sdp}{\times\kern-.2em\vrule height1.1ex depth-.05ex}
\newcommand{\epi}{\lra \kern-.8em\ra}

 \normalbaselines
\begin{document}
\renewcommand{\thefootnote}{\alph{footnote}}
\setcounter{footnote}{-1} \footnote{* Corresponding author.}
\setcounter{footnote}{-1} 
\renewcommand{\thefootnote}{\alph{footnote}}
\setcounter{footnote}{-1} 

\maketitle

\begin{abstract}
We give a unified generalization of Dugas' construction on stable auto-equivalences of Morita type from local symmetric algebras to arbitrary symmetric algebras. For group algebras $kP$ of $p$-groups in characteristic $p$, we recover all the stable auto-equivalences corresponding to endo-trivial modules over $kP$ except that $P$ is generalized quaternion of order $2^m$. Moreover, we give many examples of stable auto-equivalences of Morita type for non-local symmetric algebras.
\end{abstract}	

\section{Introduction}

In \cite{Dugas2016}, Dugas gave two methods to construct stable auto-equivalences (of Morita type) for (finite dimensional) local symmetric algebras.  One of particular interests is that such stable auto-equivalences are often not induced by auto-equivalences of the derived category.

The first construction is given as follows.

Let $A$ be an elementary local symmetric $k$-algebra, let $x\in A$ be a nilpotent element. Set $R=k[x]\cong k[X]/(X^m)$ for some integer $m\geq 2$ and $T_A=k\otimes_RA\cong A/xA$. Suppose that $_RA$ and $A_R$ are free modules and that $\underline{\mathrm{End}}_A(T)\cong k[\psi]/(\psi^{2})$, where $\psi$ is an endomorphism of $T$ induced by multiplying some $y\in A$. (As Dugas pointed out that the algebra $\underline{\mathrm{End}}_A(T)$ has a periodic bimodule free resolution of period 2.) Let $C_\mu$ be the kernel of the multiplication map $\mu: A\otimes_RA\rightarrow A$. Then $-\otimes_{A}C_\mu:\underline{\mathrm{mod}}$-$A\rightarrow\underline{\mathrm{mod}}$-$A$ is a stable auto-equivalence of $A$.

Note that $\Omega_{A^e}^{-1}(C_\mu)\cong Cone(\mu)$ in $\underline{\mathrm{mod}}$-$A^e$ and Dugas called the stable auto-equivalence $-\otimes_{A}\Omega_{A^e}^{-1}(C_\mu)$ as a spherical stable twist which is analogous to spherical twist constructed on the derived category by Seidel and Thomas. Under the more general condition $\underline{\mathrm{End}}_A(T)\cong k[\psi]/(\psi^{n+1})$ for some $n\geq 1$, Dugas gave a second construction using a double cone construction, and the induced stable auto-equivalence is called $\mathbb{P}_n$-stable twist since it is analogous to $\mathbb{P}_n$-twist on the derived category of coherent sheaves on a variety by Huybrechts and Thomas.

For group algebras of $p$-groups in characteristic $p$, Dugas recovered many of the stable auto-equivalences corresponding to endo-trivial modules. He also obtained stable auto-equivalences for local algebras of dihedral and semi-dihedral type, which are not group algebras.

In this note, we give a unified generalization of Dugas' construction by greatly relaxing the conditions on both $A$ and $R$ and by adding a new subalgebra $B$ of $A$. The main idea is as follows. For a symmetric $k$-algebra $A$, consider a triple $(A,R,B)$, where $R$, $B$ are subalgebras of $A$ such that $R$ is also symmetric and $B$ (as a $B$-$B$-bimodule) has a periodic free resolution of period $q$. Then, under some commutativity assumptions between $R$, $B$ and $A$, we may construct a complex of left-right projective $A$-$A$-bimodules. Using this complex, we can construct a left-right projective $A$-$A$-bimodule $M_q$ using a multiple cone construction such that the functor $-\otimes_{A}M_q$ induces a stable auto-equivalence of $A$. The main results are Theorem \ref{auto-equivalence} and Theorem \ref{auto-equivalence 2}.

Our construction generalizes Dugas' construction in three ways. Firstly, we dropped the condition that the algebra $A$ is local. Secondly, we don't request the subalgebra $R$ to be local or Nakayama. Thirdly, we use a subalgebra $B$ of $A$ to replace $\underline{\mathrm{End}}_A(T)$ in Dugas' construction, which is more flexible. For a connection between $B$ and $\underline{\mathrm{End}}_A(T)$, we refer to Remark \ref{comparison} below.

For group algebras $kP$ of $p$-groups in characteristic $p$, we recover all the stable auto-equivalences of $kP$ corresponding to endo-trivial modules except that $P$ is generalized quaternion of order $2^m$, see Proposition \ref{generator of T(P)}. Moreover, we can construct many examples of stable auto-equivalences of Morita type (which are not induced by derived equivalences in general) for non-local symmetric algebras, see Section 6.

Our discussion is also related to construct stable equivalences between different algebras. In particular, we will use a method in \cite{LX2007}, which gives a way to construct new stable equivalence between non-Morita equivalent algebras from a given stable auto-equivalence.

This paper is organized as follows. In Section 2, we state some general results on triangulated functors, in particular we recall some results that are useful in establishing that a given triangulated functor is an equivalence. We give the constructions of stable auto-equivalences for (not necessarily local) symmetric algebras in Section 3 and Section 4. We show in Section 5 that our construction recovers all the stable auto-equivalences corresponding to endo-trivial modules over a finite $p$-group algebra $kP$ when $P$ is not generalized quaternion of order $2^m$. In Section 6, we construct various examples of stable auto-equivalences for non-local symmetric algebras.

\section*{Data availability} The datasets generated during the current study are available from the corresponding author on reasonable request.

\section{Preliminary}

Throughout this section, let $k$ be a field and let $\mathscr{T}$ be a Hom-finite triangulated $k$-category with suspension $[1]$. A typical example of this kind of triangulated $k$-category is the stable category $\underline{\mathrm{mod}}$-$A$ of finite-dimensional right $A$-modules, where $A$ is a finite-dimensional self-injective $k$-algebra. Note that the suspension in  $\underline{\mathrm{mod}}$-$A$ is given by the cosyzygy functor $\Omega_{A}^{-1}$ and $\underline{\mathrm{mod}}$-$A$ has a Serre functor $\nu_A\Omega_{A}$, where $\nu_A$ is the Nakayama functor.

We have the following interesting result on triangulated functor.

\begin{Lem}\label{block}
Let $\mathscr{T'}$ and $\mathscr{T}_1$, $\cdots$, $\mathscr{T}_n$ be indecomposable (Hom-finite) Krull-Schmidt triangulated $k$-categories and let $\mathscr{T}=\mathscr{T}_1\times\cdots\times\mathscr{T}_n$. Let $F:\mathscr{T'}\rightarrow \mathscr{T}$ be a fully faithful triangulated functor, which maps some nonzero object $X$ of $\mathscr{T'}$ to an object of $\mathscr{T}_1$. Then the image of $F$ is in $\mathscr{T}_1$.
\end{Lem}

\begin{proof}
Since $\mathscr{T'}$ and $\mathscr{T}$ are Krull-Schmidt and $F$ is fully faithful, $F$ sends each indecomposable object $Y$ of $\mathscr{T'}$ to an indecomposable object $FY$ of $\mathscr{T}$, therefore $FY\in\mathscr{T}_i$ for some $i$. Let $\mathscr{C}_1$ (resp. $\mathscr{C}_2$) be the full subcategory of $\mathscr{T'}$ which is formed by the objects $Z$ such that $FZ\in\mathscr{T}_1$ (resp. $FZ\in\mathscr{T}_2\times\cdots\times\mathscr{T}_n$). For each object $Z$ of $\mathscr{T'}$, let $Z_i$ be the direct sum of indecomposable summands of $Z$ which belong to $\mathscr{C}_i$, $i=1$, $2$. Then $Z=Z_1\oplus Z_2$ with $Z_i\in\mathscr{C}_i$. For every pair of objects $A_i\in\mathscr{C}_i$ and for each $n\in\mathbb{Z}$, since $FA_1\in\mathscr{T}_1$ and $(FA_2)[n]\in\mathscr{T}_2\times\cdots\times\mathscr{T}_n$, $\mathscr{T'}(A_1,A_2[n])\cong\mathscr{T}(FA_1,(FA_2)[n])=0$. Since $\mathscr{T'}$ is indecomposable, either $\mathscr{C}_1$ or $\mathscr{C}_2$ is zero. Since $0\neq X\in\mathscr{C}_1$, $\mathscr{C}_2$ must be zero. Therefore $\mathscr{C}_1=\mathscr{T'}$.
\end{proof}

\begin{Rem1}\label{block-stable-category}
We will use Lemma \ref{block} in the following situation. Let $A$ be a self-injective $k$-algebra with a decomposition $A=A_1\times\cdots\times A_n$ into indecomposable algebras. Suppose that $M$ is a left-right projective $A$-$A$-bimodule and induces a fully faithful functor $-\otimes_AM: \underline{\mathrm{mod}}$-$A\rightarrow\underline{\mathrm{mod}}$-$A$ on stable category. Suppose that $X$ is a non-projective $A_1$-module such that $X\otimes_AM$ is a $A_i$-module for some $i$. Then $-\otimes_AM$ restricts to a fully faithful functor $\underline{\mathrm{mod}}$-$A_1\rightarrow\underline{\mathrm{mod}}$-$A_i$.
\end{Rem1}

Next we recall from \cite{B1999,Dugas2016} some general results that are useful in establishing that a given triangulated
functor is an equivalence.

Let $\mathscr{T}$ be a triangulated category and let $\mathscr{C}$ be a collection of objects in $\mathscr{T}$. For any $n\in\mathbb{Z}$, define $\mathscr{C}[n]:=\{X[n]\mid X\in\mathscr{C}\}$. Moreover, define ${\mathscr{C}}^{\perp}:=\{Y\in\mathscr{T}\mid \mathscr{T}(X,Y)=0$ for any $X\in\mathscr{C}\}$ and $\prescript{\perp}{}{\mathscr{C}}:=\{Y\in\mathscr{T}\mid \mathscr{T}(Y,X)=0$ for any $X\in\mathscr{C}\}$.

\begin{Def}\label{spanning class and strong spanning class} {\rm(\cite[Definition 2.1]{Dugas2016})}
Let $\mathscr{T}$ be a triangulated category. A collection $\mathscr{C}$ of objects in $\mathscr{T}$ is called a spanning class (resp. strong spanning class) if $(\bigcup_{n\in\mathbb{Z}}\mathscr{C}[n])^{\perp}=0$ and $\prescript{\perp}{}{(\bigcup_{n\in\mathbb{Z}}\mathscr{C}[n])}=0$ (resp. ${\mathscr{C}}^{\perp}=0$ and $\prescript{\perp}{}{\mathscr{C}}=0$).
\end{Def}

\begin{Rem1}
If $\mathscr{T}$ is a triangulated category which has a Serre functor, then for any object $X$ of $\mathscr{T}$, $\mathscr{C}=\{X\}\cup X^{\perp}$ is a strong spanning class of $\mathscr{T}$.
\end{Rem1}

\begin{Prop}\label{fully faithful} {\rm(\cite[Theorem 2.3]{B1999} and \cite[Proposition 2.2]{Dugas2016})}
Let $\mathscr{T}$ and $\mathscr{T}'$ be triangulated categories, and let $F:\mathscr{T}\rightarrow\mathscr{T}'$ be a triangulated functor with a left and a right adjoint. Then $F$ is fully faithful if and only if there exists a strong spanning class $\mathscr{C}$ of $\mathscr{T}$ such that $F$ induces isomorphisms $\mathscr{T}(X,Y[n])\rightarrow\mathscr{T}'(FX,F(Y[n]))$ for any $X$, $Y\in\mathscr{C}$ and for any $n=0$, $1$.
\end{Prop}

\begin{Prop}\label{equivalence} {\rm(\cite[Theorem 3.3]{B1999})}
Let $\mathscr{T}$ and $\mathscr{T}'$ be triangulated categories with $\mathscr{T}$ nonzero, $\mathscr{T}'$ indecomposable, and let $F:\mathscr{T}\rightarrow\mathscr{T}'$ be a fully faithful triangulated functor. Then $F$ is an equivalence of categories if and only if $F$ has a left adjoint $G$ and a right adjoint $H$ such that $H(Y)\cong 0$ implies $G(Y)\cong 0$ for any $Y\in\mathscr{T}'$.
\end{Prop}

Combining Propositions \ref{fully faithful} and \ref{equivalence} we have the following consequence for symmetric algebras (see the definition of a symmetric algebra in Section 3):

\begin{Cor}\label{from strong spanning class to equivalence}
Let $\Lambda$, $\Gamma$ be symmetric $k$-algebras such that $\Lambda$ is not semisimple and $\Gamma$ is indecomposable, and let $M$ be a left-right projective $\Lambda$-$\Gamma$-bimodule. Denote $F$ the stable functor induced by the functor $-\otimes_{\Lambda}M:\mathrm{mod}$-$\Lambda\rightarrow\mathrm{mod}$-$\Gamma$. If there exists a strong spanning class $\mathscr{C}$ of $\underline{\mathrm{mod}}$-$\Lambda$ such that for any $X$, $Y\in\mathscr{C}$ and for any $n=0$, $1$, the homomorphism $F:\underline{\mathrm{Hom}}_{\Lambda}(X,Y[n])\rightarrow\underline{\mathrm{Hom}}_{\Gamma}(FX,F(Y[n]))$ is an isomorphism, then $F$ is an equivalence.
\end{Cor}

\begin{proof}
Since $\Lambda$, $\Gamma$ are symmetric, by \cite[Lemma 3.2]{Dugas2016}, the functor $-\otimes_{\Gamma}DM:\mathrm{mod}$-$\Gamma\rightarrow\mathrm{mod}$-$\Lambda$ is both the left and the right adjoint of $-\otimes_{\Lambda}M:\mathrm{mod}$-$\Lambda\rightarrow\mathrm{mod}$-$\Gamma$. Therefore the stable functor $G:\underline{\mathrm{mod}}$-$\Gamma\rightarrow\underline{\mathrm{mod}}$-$\Lambda$ induced by $-\otimes_{\Gamma}DM$ is both the left and the right adjoint of $F$. By Proposition \ref{fully faithful}, $F$ is fully faithful. Since $\Lambda$ is not semisimple and $\Gamma$ is indecomposable, $\underline{\mathrm{mod}}$-$\Lambda$ is nonzero and $\underline{\mathrm{mod}}$-$\Gamma$ is indecomposable as a triangulated category. Then it follows from Proposition \ref{equivalence} that $F$ is an equivalence.
\end{proof}

\section{A construction of stable auto-equivalences for symmetric algebras}

In the following, unless otherwise stated, all algebras considered will be
finite dimensional unitary $k$-algebras over a field $k$, and all their modules will be finite dimensional right modules. By a subalgebra $B$ of an algebra $A$, we mean that $B$ is a subalgebra of $A$ with the same identity element.

We denote by $A^e$ the enveloping algebra of $A$, which by definition is $A^{op}\otimes_k A$. We let $D=\mathrm{Hom}_k(-,k)$ be the duality with respect to the ground field $k$.
Recall that an algebra $A$ is symmetric if $A\cong D(A)$ as right $A^e$-modules (or equivalently, as $A$-$A$-bimodules). It is well-known that symmetric algebras are self-injective algebras with identity Nakayama functors.

In this section, we make the following

{\bf Assumption 1:} Let $k$ be a field, $A$ be a symmetric $k$-algebra, $R$ be a non-semisimple symmetric $k$-subalgebra of $A$ such that $A_R$ is projective. Let $B$ be another $k$-subalgebra of $A$, such that the following conditions hold: \\ $(a)$ $br=rb$ for each $b\in B$ and $r\in R$; \\ $(b)$ $B\otimes_{k} (R/radR)\xrightarrow{\phi} (R/radR)\otimes_{R}A$, $b\otimes\overline{1}\mapsto\overline{1}\otimes b$ is an isomorphism in $\underline{\mathrm{mod}}$-$R$; \\ $(c)$ $B$ has a periodic free $B^e$-resolution, that is, there exists an exact sequence
\begin{equation} \label{eq: B-B-complex}
0\rightarrow B\xrightarrow{\delta_q} (B\otimes_{k}B)^{m_{q-1}}\xrightarrow{\delta_{q-1}}\cdots\rightarrow (B\otimes_{k}B)^{m_{1}}\xrightarrow{\delta_1} (B\otimes_{k}B)^{m_{0}}\xrightarrow{\delta_0} B\rightarrow 0
\end{equation}
of $B^e$-modules.

From now on, we fix $(A,R,B)$ as a triple of algebras satisfying Assumption 1.

\begin{Rem1}\label{T is not semisimple}
$(i)$ Let $T_{A}:=(R/radR)\otimes_{R}A_{A}\cong A/(radR)A$. Since $R$ is not semisimple, $R/radR$ is non-projective. Since $B\otimes_{k}(R/radR)\cong T_R$ in $\underline{\mathrm{mod}}$-$R$, $T_R$ is non-projective. Since $A_R$ is projective, $T_A$ is also non-projective. Moreover, it shows that $A$ is not semisimple.\\
$(ii)$ In most examples of this paper, $R$ is a subalgebra of $A$ with the property that $_{R}A_{R}\cong$ $_{R}R_{R}^{n}\oplus(R\otimes R)^{l}$ for some positive integers $n$ and $l$.\\
$(iii)$ The condition $(c)$ implies that $B$ is a self-injective algebra by \cite[Theorem 1.4]{GSS2003}.
\end{Rem1}

\begin{Rem1} \label{comparison}
Since $B\otimes_{k} (R/radR)\xrightarrow{\phi} (R/radR)\otimes_{R}A\cong A/(radR)A$, $b\otimes\overline{1}\mapsto\overline{b}$ is an isomorphism in $\underline{\mathrm{mod}}$-$R$, we have  isomorphisms
\begin{multline}B\otimes_{k}\underline{\mathrm{End}}_{R}(R/radR)\cong \underline{\mathrm{Hom}}_{R}(R/radR,B\otimes_{k}(R/radR))\cong \\ \underline{\mathrm{Hom}}_{R}(R/radR,A/(radR)A)\cong\underline{\mathrm{End}}_{A}(A/(radR)A),
\end{multline}
where the last isomorphism is induced from the adjoint isomorphism given by the adjoint pair $(F,G)$, where $F$ (resp. $G$) is the stable functor $\underline{\mathrm{mod}}$-$R\rightarrow\underline{\mathrm{mod}}$-$A$ (resp. $\underline{\mathrm{mod}}$-$A\rightarrow\underline{\mathrm{mod}}$-$R$) induced from the induction functor $-\otimes_{R}A$ (resp. restriction functor $-\otimes_{A}A_{R}$). Moreover, it can be shown that the composition of these isomorphisms is a $k$-algebra isomorphism from $B\otimes_{k}\underline{\mathrm{End}}_{R}(R/radR)$ to $\underline{\mathrm{End}}_{A}(A/(radR)A)$. Especially, if $R$ is an elementary local symmetric $k$-algebra, then our subalgebra $B$ is isomorphic to $\underline{\mathrm{End}}_{A}(T)=\underline{\mathrm{End}}_{A}(A/(radR)A)$, which give the connection between our construction and Dugas' construction.
\end{Rem1}

\begin{Rem1}
Since $A$ is symmetric, $\prescript{}{A}{A}$ is isomorphic to $D(A_{A})$ as $A$-modules, and $\prescript{}{R}{A}$ is isomorphic to $D(A_{R})$ as $R$-modules. Since $A_R$ is projective and $R$ is self-injective, $A_R$ is injective and therefore $\prescript{}{R}{A}\cong D(A_R)$ is projective.
\end{Rem1}

\medskip
Let $\mathrm{lrp}(A)$ be the category of left-right projective $A$-$A$-bimodules, and let $\underline{\mathrm{lrp}}(A)$ be the stable category of $\mathrm{lrp}(A)$ obtained by factoring out the morphisms that factor through a projective $A^e$-module. Since $A^e$ is self-injective (even symmetric), $\underline{\mathrm{lrp}}(A)$ becomes a triangulated category. Let $\mathrm{sum}$-$B^e$ be the full subcategory of $\mathrm{mod}$-$B^e$ consists of finite direct sum of copies of $B\otimes_{k}B$. For each $B^e$-module homomorphism $f:B\otimes_{k}B\rightarrow B\otimes_{k}B$, $1\otimes 1\mapsto \sum b_i\otimes b'_i$, applies the functor $A\otimes_{B}-\otimes_{B}A$, we have an $A^e$-homomorphism $\widetilde{f}:A\otimes_{k}A\rightarrow A\otimes_{k}A$, $1\otimes 1\mapsto \sum b_i\otimes b'_i$. Since $\widetilde{f}$ is induced from a $B^e$-homomorphism and the elements of $B$ commute with the elements of $R$ under multiplication, $\widetilde{f}$ induces an $A^e$-homomorphism $H(f):A\otimes_{R}A\rightarrow A\otimes_{R}A$, which makes the diagram
$$\xymatrix{
B\otimes_{k}B \ar[r]^{f}\ar[d] & B\otimes_{k}B \ar[d] \\
A\otimes_{R}A \ar[r]^{H(f)} & A\otimes_{R}A \\
}$$
commutes. In general, for each $B^e$-homomorphism $f:(B\otimes_{k}B)^n\rightarrow (B\otimes_{k}B)^m$ in $\mathrm{sum}$-$B^e$, let $H(f)$ be the unique $A^e$-homomorphism $(A\otimes_{R}A)^n\rightarrow (A\otimes_{R}A)^m$ such that the diagram
$$\xymatrix{
(B\otimes_{k}B)^n \ar[r]^{f}\ar[d] & (B\otimes_{k}B)^m \ar[d] \\
(A\otimes_{R}A)^n \ar[r]^{H(f)} & (A\otimes_{R}A)^m \\
}$$
commutes, where the vertical morphisms are the obvious morphisms. Then we have defined a functor $H:\mathrm{sum}$-$B^e\rightarrow\mathrm{lrp}(A)$.

Applying $H$ to the complex $(B\otimes_{k}B)^{m_{q-1}}\xrightarrow{\delta_{q-1}}\cdots\rightarrow (B\otimes_{k}B)^{m_{1}}\xrightarrow{\delta_{1}} (B\otimes_{k}B)^{m_{0}}$ in Equation (\ref{eq: B-B-complex}) we get a complex $$(A\otimes_{R}A)^{m_{q-1}}\xrightarrow{d_{q-1}}\cdots\rightarrow (A\otimes_{R}A)^{m_{1}}\xrightarrow{d_1} (A\otimes_{R}A)^{m_{0}}.$$
 Let $\widetilde{d_0}$ be the composition $(A\otimes_{k}A)^{m_{0}}\xrightarrow{A\otimes_{B}\delta_0\otimes_{B}A}A\otimes_{B}A\xrightarrow{\mu}A$, where $\mu$ is the morphism given by multiplication. Since the elements of $B$ commute with the elements of $R$ under multiplication, $\widetilde{d_0}$ induces an $A^e$-homomorphism $(A\otimes_{R}A)^{m_{0}}\xrightarrow{d_0} A$. It can be shown that $d_0d_1=0$, so the sequence
\begin{equation}\label{eq: A-A-complex}
(A\otimes_{R}A)^{m_{q-1}}\xrightarrow{d_{q-1}}\cdots\rightarrow (A\otimes_{R}A)^{m_{1}}\xrightarrow{d_{1}} (A\otimes_{R}A)^{m_{0}}\xrightarrow{d_{0}} A
\end{equation}
 is again a complex.

\begin{Lem}\label{from complex to triangles}
There exist triangles
$$M_1 \xrightarrow{\underline{i_1}}(A\otimes_{R}A)^{m_{0}}\xrightarrow{\underline{d_{0}}} A\rightarrow,$$
$$M_2 \xrightarrow{\underline{i_2}}(A\otimes_{R}A)^{m_{1}}\xrightarrow{\underline{f_{1}}} M_1 \rightarrow,$$
$$\cdots,$$
$$M_q \xrightarrow{\underline{i_q}}(A\otimes_{R}A)^{m_{q-1}}\xrightarrow{\underline{f_{q-1}}} M_{q-1} \rightarrow$$
in the triangulated category $\underline{\mathrm{lrp}}(A)$ such that $i_p f_p =d_p$ for $1\leq p\leq q-1$.
\end{Lem}

\begin{proof}
Let $i_1:M_1\rightarrow(A\otimes_{R}A)^{m_{0}}$ be the kernel of $d_0:(A\otimes_{R}A)^{m_{0}}\rightarrow A$. Since $d_0$ is surjective, $0\rightarrow M_1\xrightarrow{i_1}(A\otimes_{R}A)^{m_{0}}\xrightarrow{d_0} A\rightarrow 0$ is an exact sequence, which induces a triangle $M_1 \xrightarrow{\underline{i_1}}(A\otimes_{R}A)^{m_{0}}\xrightarrow{\underline{d_{0}}} A\rightarrow$ in $\underline{\mathrm{lrp}}(A)$. Since $d_0 d_1=0$, there exists a morphism $f_1:(A\otimes_{R}A)^{m_{1}}\rightarrow M_1$ such that $d_1=i_1 f_1$. Let $v_1:P_1\rightarrow M_1$ be the projective cover of $M_1$ as an $A^e$-module, and let
\begin{math}
\left[ \begin{smallmatrix}
i_2 \\ u_1
\end{smallmatrix} \right] :M_2\rightarrow (A\otimes_{R}A)^{m_{1}}\oplus P_1
\end{math}
be the kernel of
\begin{math}
\left[ \begin{smallmatrix}
f_1 & v_1
\end{smallmatrix} \right] :(A\otimes_{R}A)^{m_{1}}\oplus P_1\rightarrow M_1
\end{math}.
Since the morphism
\begin{math}
\left[ \begin{smallmatrix}
f_1 & v_1
\end{smallmatrix} \right]
\end{math}
is surjective, the short exact sequence
\begin{math}
0\rightarrow M_2\xrightarrow{\left[ \begin{smallmatrix}
i_2 \\ u_1
\end{smallmatrix} \right]} (A\otimes_{R}A)^{m_{1}}\oplus P_1\xrightarrow{\left[ \begin{smallmatrix}
f_1 & v_1
\end{smallmatrix} \right]} M_1\rightarrow 0
\end{math}
induces a triangle $M_2 \xrightarrow{\underline{i_2}}(A\otimes_{R}A)^{m_{1}}\xrightarrow{\underline{f_{1}}} M_1 \rightarrow$ in $\underline{\mathrm{lrp}}(A)$.
Since $i_1 f_1 d_2=d_1 d_2=0$ and $i_1$ is injective, $f_1 d_2=0$. Since the morphism
\begin{math}
\left[ \begin{smallmatrix}
d_2 \\ 0
\end{smallmatrix} \right]: (A\otimes_{R}A)^{m_{2}}\rightarrow (A\otimes_{R}A)^{m_{1}}\oplus P_1
\end{math}
satisfies
\begin{math}
\left[ \begin{smallmatrix}
f_1 & v_1
\end{smallmatrix} \right]\left[ \begin{smallmatrix}
d_2 \\ 0
\end{smallmatrix} \right]=f_1 d_2=0
\end{math},
there exists a morphism $f_2:(A\otimes_{R}A)^{m_{2}}\rightarrow M_2$ such that $d_2=i_2 f_2$ and $u_1 f_2=0$.

Using the same method, we can construct morphisms $i_p:M_p \rightarrow (A\otimes_{R}A)^{m_{p-1}}$ for $1\leq p\leq q$, and morphisms $f_{p'}:(A\otimes_{R}A)^{m_{p'}}\rightarrow M_{p'}$, $u_{p'}:M_{p'+1}\rightarrow P_{p'}$, $v_{p'}:P_{p'}\rightarrow M_{p'}$ for $1\leq p'\leq q-1$ with $P_{p'}$ projective as $A^e$-modules, such that the following conditions hold: \\ $(i)$ $i_p f_p =d_p$ for $1\leq p\leq q-1$; \\ $(ii)$ $u_p f_{p+1}=0$ for $1\leq p\leq q-2$; \\ $(iii)$ $0\rightarrow M_1\xrightarrow{i_1}(A\otimes_{R}A)^{m_{0}}\xrightarrow{d_0} A\rightarrow 0$ and
\begin{math}
0\rightarrow M_{p+1}\xrightarrow{\left[ \begin{smallmatrix}
i_{p+1} \\ u_p
\end{smallmatrix} \right]} (A\otimes_{R}A)^{m_{p}}\oplus P_p\xrightarrow{\left[ \begin{smallmatrix}
f_p & v_p
\end{smallmatrix} \right]} M_p\rightarrow 0
\end{math}
are short exact sequences for $1\leq p\leq q-1$.

Since each $P_p$ is a projective $A^e$-module, these short exact sequences induce triangles
$$M_1 \xrightarrow{\underline{i_1}}(A\otimes_{R}A)^{m_{0}}\xrightarrow{\underline{d_{0}}} A\rightarrow,$$
$$M_2 \xrightarrow{\underline{i_2}}(A\otimes_{R}A)^{m_{1}}\xrightarrow{\underline{f_{1}}} M_1 \rightarrow,$$
$$\cdots,$$
$$M_q \xrightarrow{\underline{i_q}}(A\otimes_{R}A)^{m_{q-1}}\xrightarrow{\underline{f_{q-1}}} M_{q-1} \rightarrow$$
in $\underline{\mathrm{lrp}}(A)$.
\end{proof}

\begin{Thm}\label{auto-equivalence}
Let $(A,R,B)$ be the triple that satisfies Assumption 1. If $M_q$ is the $A$-$A$-bimodule defined in Lemma \ref{from complex to triangles}, then $-\otimes_{A}M_q:\underline{\mathrm{mod}}$-$A\rightarrow\underline{\mathrm{mod}}$-$A$ is a stable auto-equivalence of $A$.
\end{Thm}

\begin{proof}
Let $F=-\otimes_{R}A_{A}$ and $G=-\otimes_{A}A_{R}$ be the induction and the restriction functors respectively. Since $A$ and $R$ are symmetric and $\prescript{}{R}{A}_{A}$ is left-right projective, both $(F,G)$ and $(G,F)$ are adjoint pairs. Since both $F$ and $G$ map projectives to projectives, they induce stable functors (which are also denoted by $F$ and $G$). Moreover, $G$ is both the left and the right adjoint of $F$ as stable functors. Let $T_{A}=F(R/radR)=(R/radR)\otimes_{R}A_{A}\cong A/(radR)A$. According to Remark \ref{T is not semisimple}, $T_A$ is a nonzero object in $\underline{\mathrm{mod}}$-$A$. Since the elements of $B$ commute with the elements of $R$ under multiplication, $T\cong A/(radR)A$ becomes a $B$-$A$-bimodule.

Under the above notations, we now prove that $-\otimes_{A}M_q:\underline{\mathrm{mod}}$-$A\rightarrow\underline{\mathrm{mod}}$-$A$ is a stable auto-equivalence of $A$. We will consider two cases.

\medskip
{\it Case 1: Assume that $A$ (as an algebra) is indecomposable.}

Choose a strong spanning class $\mathscr{C}=\{T\}\cup T^{\perp}$ of $\underline{\mathrm{mod}}$-$A$, where $T^{\perp}=\{X\in\underline{\mathrm{mod}}$-$A\mid \underline{\mathrm{Hom}}_{A}(T,X)=0\}$. According to Corollary \ref{from strong spanning class to equivalence}, it suffices to show that $-\otimes_{A}M_q$ induces bijections between $\underline{\mathrm{Hom}}_{A}(X,Y[i])$ and $\underline{\mathrm{Hom}}_{A}(X\otimes_{A}M_q,(Y[i])\otimes_{A}M_q)$ for all $X$, $Y\in\mathscr{C}$ and for all $i=0$, $1$. We will divide the proof of Case 1 into four steps.

\medskip
{\it Step 1.1: To show that $-\otimes_{A}M_q$ induces a bijection between $\underline{\mathrm{Hom}}_{A}(T,T)$ and $\underline{\mathrm{Hom}}_{A}(T\otimes_{A}M_q,T\otimes_{A}M_q)$.}

Since $\phi:B\otimes_{k} (R/radR)\rightarrow A/(radR)A$, $b\otimes\overline{1}\mapsto\overline{b}$ is an isomorphism in $\underline{\mathrm{mod}}$-$R$, $\phi\otimes 1:B\otimes_{k}T\cong B\otimes_{k} (R/radR)\otimes_{R}A\rightarrow A/(radR)A\otimes_{R}A=T\otimes_{R}A$ is an isomorphism in $\underline{\mathrm{mod}}$-$A$. Applying the functors $-\otimes_{B}T_{A}$ and $T\otimes_{A}-$ to the complex $0\rightarrow B\xrightarrow{\delta_q} (B\otimes_{k}B)^{m_{q-1}}\xrightarrow{\delta_{q-1}}\cdots\rightarrow (B\otimes_{k}B)^{m_{1}}\xrightarrow{\delta_1} (B\otimes_{k}B)^{m_{0}}\xrightarrow{\delta_0} B\rightarrow 0$ and the complex $(A\otimes_{R}A)^{m_{q-1}}\xrightarrow{d_{q-1}}\cdots\rightarrow (A\otimes_{R}A)^{m_{1}}\xrightarrow{d_{1}} (A\otimes_{R}A)^{m_{0}}\xrightarrow{d_{0}} A$ respectively, we get a commutative diagram in $\mathrm{mod}$-$A$:

$$\xymatrix@R=2pc@C=2pc@W=0mm{
 0 \ar[r] & T \ar[r]^(0.3){\delta_{q}\otimes 1} & (B\otimes_{k}T)^{m_{q-1}}\ar[r]^(0.7){\delta_{q-1}\otimes 1}\ar[d]_{(\phi\otimes 1)^{m_{q-1}}} & \cdots\ar[r] & (B\otimes_{k}T)^{m_{1}}\ar[r]^{\delta_{1}\otimes 1}\ar[d]_{(\phi\otimes 1)^{m_{1}}} & (B\otimes_{k}T)^{m_{0}}\ar[r]^(0.7){\delta_{0}\otimes 1}\ar[d]_{(\phi\otimes 1)^{m_{0}}} & T\ar[r]\ar@{=}[d] & 0 \\
  & & (T\otimes_{R}A)^{m_{q-1}}\ar[r]^(0.7){1\otimes d_{q-1}} & \cdots\ar[r] & (T\otimes_{R}A)^{m_{1}}\ar[r]^{1\otimes d_{1}} & (T\otimes_{R}A)^{m_{0}}\ar[r]^(0.7){1\otimes d_{0}} & T &  \\
 }$$

Since $0\rightarrow B\xrightarrow{\delta_q} (B\otimes_{k}B)^{m_{q-1}}\xrightarrow{\delta_{q-1}}\cdots\rightarrow (B\otimes_{k}B)^{m_{1}}\xrightarrow{\delta_1} (B\otimes_{k}B)^{m_{0}}\xrightarrow{\delta_0} B\rightarrow 0$ is split exact as a complex of right $B$-modules, the first row of this commutative diagram is also split exact. Therefore we have split exact sequences $0\rightarrow K_1\xrightarrow{j_1}(B\otimes_{k}T)^{m_{0}}\xrightarrow{\delta_{0}\otimes 1}T\rightarrow 0$, $0\rightarrow K_2\xrightarrow{j_2}(B\otimes_{k}T)^{m_{1}}\xrightarrow{p_1}K_1\rightarrow 0$, $\cdots$, $0\rightarrow K_{q-1}\xrightarrow{j_{q-1}}(B\otimes_{k}T)^{m_{q-2}}\xrightarrow{p_{q-2}}K_{q-2}\rightarrow 0$, $0\rightarrow T\xrightarrow{\delta_{q}\otimes 1}(B\otimes_{k}T)^{m_{q-1}}\xrightarrow{p_{q-1}}K_{q-1}\rightarrow 0$ in $\mathrm{mod}$-$A$ such that $j_l p_l=\delta_{l}\otimes 1$ for $1\leq l\leq q-1$.

There is a commutative diagram
$$\xymatrix@R=2pc@C=2pc@W=0mm{
 K_1\ar[r]^(0.3){\underline{j_1}} & (B\otimes_{k}T)^{m_{0}}\ar[r]^(0.7){\underline{\delta_{0}\otimes 1}}\ar[d]_{\underline{(\phi\otimes 1)^{m_{0}}}} &T\ar[r]\ar@{=}[d] & \\
 T\otimes_{A}M_1 \ar[r]^{\underline{1\otimes i_1}} & (T\otimes_{R}A)^{m_{0}}\ar[r]^(0.7){\underline{1\otimes d_{0}}} & T\ar[r] & \\
 }$$
in $\underline{\mathrm{mod}}$-$A$, where its two rows are triangles and $\underline{(\phi\otimes 1)^{m_{0}}}$ is an isomorphism in $\underline{\mathrm{mod}}$-$A$. Therefore we have an isomorphism $\underline{\alpha_1}:K_1\rightarrow T\otimes_{A}M_1$ in $\underline{\mathrm{mod}}$-$A$ such that $\underline{(\phi\otimes 1)^{m_{0}}j_1}=\underline{(1\otimes i_{1})\alpha_1}$. Since $\underline{j_1}$ is a split monomorphism in $\underline{\mathrm{mod}}$-$A$, so does $\underline{1\otimes i_{1}}$. Since \begin{multline} \underline{(1\otimes i_{1}) \alpha_1 p_1}=\underline{(\phi\otimes 1)^{m_{0}}j_1 p_1}=\underline{(\phi\otimes 1)^{m_{0}}(\delta_{1}\otimes 1)}= \\ \underline{(1\otimes d_{1})(\phi\otimes 1)^{m_{1}}}=\underline{(1\otimes i_1) (1\otimes f_1) (\phi\otimes 1)^{m_{1}}} \end{multline} in $\underline{\mathrm{mod}}$-$A$ and since $\underline{1\otimes i_{1}}$ is a split monomorphism in $\underline{\mathrm{mod}}$-$A$, we have $\underline{\alpha_1 p_1}=\underline{(1\otimes f_1) (\phi\otimes 1)^{m_{1}}}$ in $\underline{\mathrm{mod}}$-$A$. Then we have a commutative diagram
$$\xymatrix@R=2pc@C=2pc@W=0mm{
 K_2\ar[r]^(0.3){\underline{j_2}} & (B\otimes_{k}T)^{m_{1}}\ar[r]^(0.7){\underline{p_1}}\ar[d]_{\underline{(\phi\otimes 1)^{m_{1}}}} &K_1\ar[r]\ar[d]^{\underline{\alpha_1}} & \\
 T\otimes_{A}M_2 \ar[r]^{\underline{1\otimes i_2}} & (T\otimes_{R}A)^{m_{1}}\ar[r]^{\underline{1\otimes f_{1}}} & T\otimes_{A}M_1\ar[r] & \\
 }$$
in $\underline{\mathrm{mod}}$-$A$, whose rows are triangles and vertical morphisms are isomorphisms. So we have an isomorphism $\underline{\alpha_2}:K_2\rightarrow T\otimes_{A}M_2$ in $\underline{\mathrm{mod}}$-$A$ such that $\underline{(\phi\otimes 1)^{m_{1}}j_2}=\underline{(1\otimes i_{2})\alpha_2}$. Inductively, we have isomorphisms $\underline{\alpha_l}:K_l\rightarrow T\otimes_{A}M_l$ in $\underline{\mathrm{mod}}$-$A$ for $1\leq l\leq q$ (let $K_q=T$), such that
$$\xymatrix@R=2pc@C=2pc@W=0mm{
 K_1\ar[r]^(0.3){\underline{j_1}}\ar[d]^{\underline{\alpha_1}} & (B\otimes_{k}T)^{m_{0}}\ar[r]^(0.7){\underline{\delta_{0}\otimes 1}}\ar[d]_{\underline{(\phi\otimes 1)^{m_{0}}}} &T\ar[r]\ar@{=}[d] & \\
 T\otimes_{A}M_1 \ar[r]^{\underline{1\otimes i_1}} & (T\otimes_{R}A)^{m_{0}}\ar[r]^(0.7){\underline{1\otimes d_{0}}} & T\ar[r] & \\
 }$$
is an isomorphism of triangles and
$$\xymatrix@R=2pc@C=2pc@W=0mm{
 K_{l+1}\ar[r]^(0.3){\underline{j_{l+1}}}\ar[d]^{\underline{\alpha_{l+1}}} & (B\otimes_{k}T)^{m_{l}}\ar[r]^(0.7){\underline{p_l}}\ar[d]_{\underline{(\phi\otimes 1)^{m_{l}}}} &K_l\ar[r]\ar[d]^{\underline{\alpha_l}} & \\
 T\otimes_{A}M_{l+1} \ar[r]^{\underline{1\otimes i_{l+1}}} & (T\otimes_{R}A)^{m_{l}}\ar[r]^{\underline{1\otimes f_{l}}} & T\otimes_{A}M_l\ar[r] & \\
 }$$
are isomorphisms of triangles for $1\leq l\leq q-1$ (let $j_q=\delta_{q}\otimes 1:T\rightarrow (B\otimes_{k}T)^{m_{q-1}}$).

Since $\underline{\alpha_q}:T\rightarrow T\otimes_{A}M_q$ is an isomorphism in $\underline{\mathrm{mod}}$-$A$, to show $-\otimes_{A}M_q$ induces a bijection between $\underline{\mathrm{Hom}}_{A}(T,T)$ and $\underline{\mathrm{Hom}}_{A}(T\otimes_{A}M_q,T\otimes_{A}M_q)$, it suffices to show that for each $\underline{f}\in\underline{\mathrm{End}}_{A}(T)$, the diagram
$$\xymatrix{
T\ar[r]^{\underline{f}}\ar[d]^{\underline{\alpha_q}} & T\ar[d]^{\underline{\alpha_q}} \\
T\otimes_{A}M_q\ar[r]^{\underline{f\otimes 1}} & T\otimes_{A}M_q \\
}$$
is commutative. We have an isomorphism $\underline{\mathrm{End}}_{A}(T)\cong\underline{\mathrm{Hom}}_{R}(R/radR,T_{R})\cong\underline{\mathrm{Hom}}_{R}(R/radR,B\otimes_{k}(R/radR))$, where the second isomorphism is induced from the isomorphism $\phi:B\otimes_{k} (R/radR)\rightarrow A/(radR)A$, $b\otimes\overline{1}\mapsto\overline{b}$ in $\underline{\mathrm{mod}}$-$R$. For $\underline{f}\in\underline{\mathrm{End}}_{A}(T)$, suppose the isomorphism $\underline{\mathrm{End}}_{A}(T)\rightarrow\underline{\mathrm{Hom}}_{R}(R/radR,B\otimes_{k}(R/radR))$ maps $\underline{f}$ to $\underline{g}$, where $g(\overline{1})=\sum_{j}\beta_j\otimes\overline{r_j}$ with $\beta_j\in B$, $r_j\in R$. Then $\underline{f}=\underline{h}$, where $h:T_{A}\rightarrow T_{A}$, $\overline{1}\mapsto \overline{\sum_{j}\beta_j r_j}$. Consider the diagram
$$\xymatrix@R=3pc@C=3pc@W=0mm{
&T\ar[dl]_{\underline{h}}\ar[dd]_{\underline{\alpha_q}}\ar[ddrrrr]^(0.7){\underline{(\phi\otimes 1)^{m_{q-1}}(\delta_{q}\otimes 1)}}&&&& \\
T\ar[dd]_{\underline{\alpha_q}}\ar[ddrrrr]^(0.7){\underline{(\phi\otimes 1)^{m_{q-1}}(\delta_{q}\otimes 1)}}&&&&& \\
&T\otimes_{A}M_q\ar[dl]^{\underline{h\otimes 1}}\ar[rrrr]^{\underline{1\otimes i_q}}&&&&(T\otimes_{R}A)^{m_{q-1}}\ar[dl]^{\underline{(h\otimes 1)^{m_{q-1}}}} \\
T\otimes_{A}M_q\ar[rrrr]^{\underline{1\otimes i_q}}&&&&(T\otimes_{R}A)^{m_{q-1}}& \\
}$$
$$\mbox{Figure }1$$
in $\underline{\mathrm{mod}}$-$A$, where $(\phi\otimes 1)^{m_{q-1}}(\delta_{q}\otimes 1)$ denotes the composition $T\xrightarrow{\delta_{q}\otimes 1}(B\otimes_{k}T)^{m_{q-1}}\xrightarrow{(\phi\otimes 1)^{m_{q-1}}}(T\otimes_{R}A)^{m_{q-1}}$. Since
$$\xymatrix@R=2pc@C=2pc@W=0mm{
T\ar[r]^(0.3){\underline{\delta_{q}\otimes 1}}\ar[d]^{\underline{\alpha_{q}}} & (B\otimes_{k}T)^{m_{q-1}}\ar[r]^(0.7){\underline{p_{q-1}}}\ar[d]_{\underline{(\phi\otimes 1)^{m_{q-1}}}} &K_{q-1}\ar[r]\ar[d]^{\underline{\alpha_{q-1}}} & \\
T\otimes_{A}M_{k} \ar[r]^(0.4){\underline{1\otimes i_{q}}} & (T\otimes_{R}A)^{m_{q-1}}\ar[r]^(0.52){\underline{1\otimes f_{q-1}}} & T\otimes_{A}M_{q-1}\ar[r] & \\
}$$
is an isomorphism of triangles in $\underline{\mathrm{mod}}$-$A$, and since $\underline{\delta_{q}\otimes 1}$ is a split monomorphism in $\underline{\mathrm{mod}}$-$A$,  $\underline{1\otimes i_{q}}$ is also a split monomorphism in $\underline{\mathrm{mod}}$-$A$. Since the bottom face, the front face, the back face of Figure $1$ are commutative, and since $\underline{1\otimes i_{q}}$ is a split monomorphism, to show the left face of Figure $1$ commutes, it suffices to show the diagram
$$\xymatrix@R=3pc@C=4pc@W=0mm{
T\ar[r]^{\underline{h}}\ar[d]_{\underline{\delta_{q}\otimes 1}}&T\ar[d]^{\underline{\delta_{q}\otimes 1}} \\
(B\otimes_{k}T)^{m_{q-1}}\ar[d]_{\underline{(\phi\otimes 1)^{m_{q-1}}}}&(B\otimes_{k}T)^{m_{q-1}}\ar[d]^{\underline{(\phi\otimes 1)^{m_{q-1}}}} \\
(T\otimes_{R}A)^{m_{q-1}}\ar[r]^{\underline{(h\otimes 1)^{m_{q-1}}}}&(T\otimes_{R}A)^{m_{q-1}} \\
}$$
$$\mbox{Figure }2$$
is commutative in $\underline{\mathrm{mod}}$-$A$.

Since $\delta_q:B\rightarrow(B\otimes_{k}B)^{m_{q-1}}$ is a $B^e$-homomorphism, we may write $\delta_q$ as $(\delta_{q}^{1},\cdots,\delta_{q}^{m_{q-1}})'$, where $\delta_{q}^{i}:B\rightarrow B\otimes_{k}B$, $1\mapsto \sum_{l} b_{il}\otimes b'_{il}$ for $1\leq i\leq m_{q-1}$. To show that the diagram in Figure $2$ commutes, it suffices to show for each $1\leq i\leq m_{q-1}$, the diagram
$$\xymatrix@R=3pc@C=4pc@W=0mm{
T\ar[r]^{h}\ar[d]_{\delta_{q}^{i}\otimes 1} & T\ar[d]^{\delta_{q}^{i}\otimes 1} \\
B\otimes_{k}T\ar[d]_{\phi\otimes 1} & B\otimes_{k}T\ar[d]^{\phi\otimes 1} \\
T\otimes_{R}A\ar[r]^{h\otimes 1} & T\otimes_{R}A \\
}$$
$$\mbox{Figure }3$$
is commutative in $\mathrm{mod}$-$A$.

For $\overline{1}\in T=A/(radR)A$, $(h\otimes 1)(\phi\otimes 1)(\delta_{q}^{i}\otimes 1)(\overline{1})=(h\otimes 1)(\phi\otimes 1)(\sum_{l} b_{il}\otimes \overline{b'_{il}})=(h\otimes 1)(\sum_{l} \overline{b_{il}}\otimes b'_{il})=\sum_{l} \overline{(\sum_{j}\beta_j r_j)b_{il}}\otimes b'_{il}=\sum_{j} (\sum_{l}\overline{\beta_j b_{il}}\otimes b'_{il})r_j$, where the last identity follows from the fact that the elements of $B$ commute with the elements of $R$ under multiplication. Moreover, $(\phi\otimes 1)(\delta_{q}^{i}\otimes 1)h(\overline{1})=(\phi\otimes 1)(\delta_{q}^{i}\otimes 1)(\overline{\sum_{j}\beta_j r_j})=(\phi\otimes 1)(\sum_{l} b_{il}\otimes \overline{b'_{il}(\sum_{j}\beta_j r_j)})=\sum_{l} \overline{b_{il}}\otimes b'_{il}(\sum_{j}\beta_j r_j)=\sum_{j}(\sum_{l} \overline{b_{il}}\otimes b'_{il}\beta_j) r_j$. Since $\delta_{q}^{i}:B\rightarrow B\otimes_{k}B$ is a $B^e$-homomorphism, $\sum_{l}\beta_j b_{il}\otimes b'_{il}=\beta_j(\sum_{l} b_{il}\otimes b'_{il})=\beta_j\delta_{q}^{i}(1)=\delta_{q}^{i}(\beta_j)=\delta_{q}^{i}(1)\beta_j=(\sum_{l} b_{il}\otimes b'_{il})\beta_j=\sum_{l} b_{il}\otimes b'_{il}\beta_j$ in $B\otimes_{k}B$. Since $\sum_{l}\overline{\beta_j b_{il}}\otimes b'_{il}\in T\otimes_{R}A$ (resp. $\sum_{l} \overline{b_{il}}\otimes b'_{il}\beta_j\in T\otimes_{R}A$) is the image of $\sum_{l}\beta_j b_{il}\otimes b'_{il}$ (resp. $\sum_{l} b_{il}\otimes b'_{il}\beta_j$) under the composition of morphisms $B\otimes_{k}B\rightarrow A\otimes_{k}A\rightarrow A\otimes_{R}A\rightarrow T\otimes_{R}A$, $\sum_{l}\overline{\beta_j b_{il}}\otimes b'_{il}=\sum_{l} \overline{b_{il}}\otimes b'_{il}\beta_j$ in $T\otimes_{R}A$. Therefore $(h\otimes 1)(\phi\otimes 1)(\delta_{q}^{i}\otimes 1)(\overline{1})=\sum_{j} (\sum_{l}\overline{\beta_j b_{il}}\otimes b'_{il})r_j=\sum_{j}(\sum_{l} \overline{b_{il}}\otimes b'_{il}\beta_j) r_j=(\phi\otimes 1)(\delta_{q}^{i}\otimes 1)h(\overline{1})$ and the diagram in Figure $3$ commutes.

\medskip
{\it Step 1.2: To show that $-\otimes_{A}M_q$ induces a bijection between $\underline{\mathrm{Hom}}_{A}(T,T[1])$ and $\underline{\mathrm{Hom}}_{A}(T\otimes_{A}M_q,T[1]\otimes_{A}M_q)$.}

Since the functor $-\otimes_{A}M_q:\underline{\mathrm{mod}}$-$A\rightarrow\underline{\mathrm{mod}}$-$A$ commutes with the functor $[1]=\Omega_{A}^{-1}:\underline{\mathrm{mod}}$-$A\rightarrow\underline{\mathrm{mod}}$-$A$ up to natural isomorphism, it suffices to show $-\otimes_{A}M_q$ induces a bijection between $\underline{\mathrm{Hom}}_{A}(\Omega_{A}T,T)$ and $\underline{\mathrm{Hom}}_{A}(\Omega_{A}T\otimes_{A}M_q,T\otimes_{A}M_q)$.

There is a commutative diagram
$$\xymatrix{
0\ar[r] & B\otimes_{k}radR\ar[r]\ar[d]^{\nu} & B\otimes_{k}R\ar[r]\ar[d]^{\mu} & B\otimes_{k}(R/radR)\ar[r]\ar[d]^{\phi} & 0 \\
0\ar[r] & (radR)A\ar[r] & A\ar[r] & A/(radR)A\ar[r] & 0 \\
}$$
in $\mathrm{mod}$-$R$ with exact rows, where $\mu$ and $\nu$ are induced by the multiplication of $A$. Since $R$ is symmetric and $A_{R}$ is projective, $\underline{\nu}=\Omega_{R}(\underline{\phi})$ is an isomorphism in $\underline{\mathrm{mod}}$-$R$. Therefore $B\otimes_{k}\Omega_{A}T=B\otimes_{k}(radR)A\cong B\otimes_{k}radR\otimes_{R}A\xrightarrow{\nu\otimes 1}(radR)A\otimes_{R}A=\Omega_{A}T\otimes_{R}A$ is an isomorphism in $\underline{\mathrm{mod}}$-$A$.

Since the elements of $B$ commute with the elements of $R$ under multiplication, $\Omega_{A}T=(radR)A$ becomes a $B$-$A$-bimodule. Applies the functors $-\otimes_{B}(\Omega_{A}T)_{A}$ and $\Omega_{A}T\otimes_{A}-$ to the complex $0\rightarrow B\xrightarrow{\delta_q} (B\otimes_{k}B)^{m_{q-1}}\xrightarrow{\delta_{q-1}}\cdots\rightarrow (B\otimes_{k}B)^{m_{1}}\xrightarrow{\delta_1} (B\otimes_{k}B)^{m_{0}}\xrightarrow{\delta_0} B\rightarrow 0$ and the complex $(A\otimes_{R}A)^{m_{q-1}}\xrightarrow{d_{q-1}}\cdots\rightarrow (A\otimes_{R}A)^{m_{1}}\xrightarrow{d_{1}} (A\otimes_{R}A)^{m_{0}}\xrightarrow{d_{0}} A$ respectively, we get a commutative diagram in $\mathrm{mod}$-$A$:

$$\xymatrix@R=2pc@C=1.4pc@W=0mm{
 0 \ar[r] & \Omega_{A}T \ar[r]^(0.3){\delta_{q}\otimes 1} & (B\otimes_{k}\Omega_{A}T)^{m_{q-1}}\ar[r]^(0.7){\delta_{q-1}\otimes 1}\ar[d]_{(\nu\otimes 1)^{m_{q-1}}} & \cdots\ar[r] & (B\otimes_{k}\Omega_{A}T)^{m_{1}}\ar[r]^{\delta_{1}\otimes 1}\ar[d]_{(\nu\otimes 1)^{m_{1}}} & (B\otimes_{k}\Omega_{A}T)^{m_{0}}\ar[r]^(0.7){\delta_{0}\otimes 1}\ar[d]_{(\nu\otimes 1)^{m_{0}}} & \Omega_{A}T\ar[r]\ar@{=}[d] & 0 \\
  & & (\Omega_{A}T\otimes_{R}A)^{m_{q-1}}\ar[r]^(0.7){1\otimes d_{q-1}} & \cdots\ar[r] & (\Omega_{A}T\otimes_{R}A)^{m_{1}}\ar[r]^{1\otimes d_{1}} & (\Omega_{A}T\otimes_{R}A)^{m_{0}}\ar[r]^(0.7){1\otimes d_{0}} & \Omega_{A}T &  \\
 }$$

By the same argument as in Step 1.1, we have isomorphisms of split triangles
$$\xymatrix@R=2pc@C=2pc@W=0mm{
 L_{l+1}\ar[r]^(0.3){\underline{\iota_{l+1}}}\ar[d]^{\underline{\beta_{l+1}}} & (B\otimes_{k}\Omega_{A}T)^{m_{l}}\ar[r]^(0.7){\underline{q_l}}\ar[d]_{\underline{(\nu\otimes 1)^{m_{l}}}} &L_l\ar[r]\ar[d]^{\underline{\beta_l}} & \\
 \Omega_{A}T\otimes_{A}M_{l+1} \ar[r]^{\underline{1\otimes i_{l+1}}} & (\Omega_{A}T\otimes_{R}A)^{m_{l}}\ar[r]^{\underline{1\otimes f_{l}}} & \Omega_{A}T\otimes_{A}M_l\ar[r] & \\
 }$$
in $\underline{\mathrm{mod}}$-$A$ for $0\leq l\leq q-1$, where $L_0=L_q=\Omega_{A}T$, $q_0=\delta_0\otimes 1:(B\otimes_{k}\Omega_{A}T)^{m_{0}}\rightarrow \Omega_{A}T$, $f_0=d_0:(A\otimes_{R}A)^{m_0}\rightarrow A$, $\iota_q=\delta_q\otimes 1:\Omega_{A}T\rightarrow(B\otimes_{k}\Omega_{A}T)^{m_{q-1}}$.

To show $-\otimes_{A}M_q$ induces a bijection between $\underline{\mathrm{Hom}}_{A}(\Omega_{A}T,T)$ and $\underline{\mathrm{Hom}}_{A}(\Omega_{A}T\otimes_{A}M_q,T\otimes_{A}M_q)$, it suffices to show that for each $\underline{f}\in\underline{\mathrm{Hom}}_{A}(\Omega_{A}T,T)$, the diagram
$$\xymatrix{
\Omega_{A}T\ar[r]^{\underline{f}}\ar[d]^{\underline{\beta_q}} & T\ar[d]^{\underline{\alpha_q}} \\
\Omega_{A}T\otimes_{A}M_q\ar[r]^{\underline{f\otimes 1}} & T\otimes_{A}M_q \\
}$$
is commutative. We have isomorphisms \begin{equation}\underline{\mathrm{Hom}}_{A}(\Omega_{A}T,T)=\underline{\mathrm{Hom}}_{A}(F(radR),T)\cong\underline{\mathrm{Hom}}_{R}(radR,T_{R}) \\
\cong\underline{\mathrm{Hom}}_{R}(radR,B\otimes_{k}(R/radR)),\end{equation} where the second isomorphism is induced from the isomorphism $\phi:B\otimes_{k} (R/radR)\rightarrow A/(radR)A$, $b\otimes\overline{1}\mapsto\overline{b}$ in $\underline{\mathrm{mod}}$-$R$. Choose a $k$-basis $x_1$, $\cdots$, $x_n$ of $B$, then each $g\in\mathrm{Hom}_{R}(radR,B\otimes_{k}(R/radR))$ can be written as a column vector $(g_1,\cdots,g_n)'$, where $g_i\in\mathrm{Hom}_{R}(radR,R/radR)$ for $1\leq i\leq n$. For $\underline{f}\in\underline{\mathrm{Hom}}_{A}(\Omega_{A}T,T)$, suppose the isomorphism $\underline{\mathrm{Hom}}_{A}(\Omega_{A}T,T)\rightarrow\underline{\mathrm{Hom}}_{R}(radR,B\otimes_{k}(R/radR))$ maps $\underline{f}$ to $\underline{g}$, where $g=(g_1,\cdots,g_n)'$ with $g_i\in\mathrm{Hom}_{R}(radR,R/radR)$. Suppose for each $r\in radR$, $g_{i}(r)=\overline{\gamma_i}$ with $\gamma_i\in R$. Then $\underline{f}=\underline{h}$, where $h\in\mathrm{Hom}_{A}(\Omega_{A}T,T)$ with $h(r)=\overline{\sum_{i=1}^{n}x_i \gamma_i}$ for each $r\in radR$. Consider the diagram
$$\xymatrix@R=3pc@C=3pc@W=0mm{
&\Omega_{A}T\ar[dl]_{\underline{h}}\ar[dd]_{\underline{\beta_q}}\ar[ddrrrr]^(0.7){\underline{(\nu\otimes 1)^{m_{q-1}}(\delta_{q}\otimes 1)}}&&&& \\
T\ar[dd]_{\underline{\alpha_q}}\ar[ddrrrr]^(0.7){\underline{(\phi\otimes 1)^{m_{q-1}}(\delta_{q}\otimes 1)}}&&&&& \\
&\Omega_{A}T\otimes_{A}M_q\ar[dl]^{\underline{h\otimes 1}}\ar[rrrr]^{\underline{1\otimes i_q}}&&&&(\Omega_{A}T\otimes_{R}A)^{m_{q-1}}\ar[dl]^{\underline{(h\otimes 1)^{m_{q-1}}}} \\
T\otimes_{A}M_q\ar[rrrr]^{\underline{1\otimes i_q}}&&&&(T\otimes_{R}A)^{m_{q-1}}& \\
}$$
$$\mbox{Figure }4$$
in $\underline{\mathrm{mod}}$-$A$, where $(\phi\otimes 1)^{m_{q-1}}(\delta_{q}\otimes 1)$ denotes the composition $T\xrightarrow{\delta_{q}\otimes 1}(B\otimes_{k}T)^{m_{q-1}}\xrightarrow{(\phi\otimes 1)^{m_{q-1}}}(T\otimes_{R}A)^{m_{q-1}}$ and $(\nu\otimes 1)^{m_{q-1}}(\delta_{q}\otimes 1)$ denotes the composition $\Omega_{A}T\xrightarrow{\delta_{q}\otimes 1}(B\otimes_{k}\Omega_{A}T)^{m_{q-1}}\xrightarrow{(\nu\otimes 1)^{m_{q-1}}}(\Omega_{A}T\otimes_{R}A)^{m_{q-1}}$. Since the bottom face, the front face, the back face of Figure $4$ are commutative, and since $\underline{1\otimes i_{q}}$ is a split monomorphism in $\underline{\mathrm{mod}}$-$A$, to show the left face of Figure $4$ commutes, it suffices to show the diagram
$$\xymatrix@R=3pc@C=4pc@W=0mm{
\Omega_{A}T\ar[r]^{\underline{h}}\ar[d]_{\underline{\delta_{q}\otimes 1}}&T\ar[d]^{\underline{\delta_{q}\otimes 1}} \\
(B\otimes_{k}\Omega_{A}T)^{m_{q-1}}\ar[d]_{\underline{(\nu\otimes 1)^{m_{q-1}}}}&(B\otimes_{k}T)^{m_{q-1}}\ar[d]^{\underline{(\phi\otimes 1)^{m_{q-1}}}} \\
(\Omega_{A}T\otimes_{R}A)^{m_{q-1}}\ar[r]^{\underline{(h\otimes 1)^{m_{q-1}}}}&(T\otimes_{R}A)^{m_{q-1}} \\
}$$
$$\mbox{Figure }5$$
is commutative in $\underline{\mathrm{mod}}$-$A$.

Since $\delta_q:B\rightarrow(B\otimes_{k}B)^{m_{q-1}}$ is a $B^e$-homomorphism, we may write $\delta_q$ as $(\delta_{q}^{1},\cdots,\delta_{q}^{m_{q-1}})'$, where $\delta_{q}^{i}:B\rightarrow B\otimes_{k}B$, $1\mapsto \sum_{l} b_{il}\otimes b'_{il}$ for $1\leq i\leq m_{q-1}$. To show the diagram in Figure $5$ commutes, it suffices to show for each $1\leq i\leq m_{q-1}$, the diagram
$$\xymatrix@R=3pc@C=4pc@W=0mm{
\Omega_{A}T\ar[r]^{h}\ar[d]_{\delta_{q}^{i}\otimes 1} & T\ar[d]^{\delta_{q}^{i}\otimes 1} \\
B\otimes_{k}\Omega_{A}T\ar[d]_{\nu\otimes 1} & B\otimes_{k}T\ar[d]^{\phi\otimes 1} \\
\Omega_{A}T\otimes_{R}A\ar[r]^{h\otimes 1} & T\otimes_{R}A \\
}$$
$$\mbox{Figure }6$$
is commutative in $\mathrm{mod}$-$A$.

For each $r\in radR\subseteq (radR)A=T$, $(h\otimes 1)(\nu\otimes 1)(\delta_{q}^{i}\otimes 1)(r)=(h\otimes 1)(\nu\otimes 1)(\sum_{l} b_{il}\otimes b'_{il}r)=(h\otimes 1)(\nu\otimes 1)(\sum_{l} b_{il}\otimes rb'_{il})=(h\otimes 1)(\sum_{l} b_{il}r\otimes b'_{il})=(h\otimes 1)(\sum_{l} rb_{il}\otimes b'_{il})=\sum_{l} \overline{(\sum_{j=1}^{n}x_j \gamma_j)b_{il}}\otimes b'_{il}=\sum_{j=1}^{n}(\sum_{l}\overline{x_jb_{il}}\otimes b'_{il})\gamma_j$ and $(\phi\otimes 1)(\delta_{q}^{i}\otimes 1)h(r)=(\phi\otimes 1)(\delta_{q}^{i}\otimes 1)(\overline{\sum_{j=1}^{n}x_j \gamma_j})=(\phi\otimes 1)(\sum_{l} b_{il}\otimes \overline{b'_{il}(\sum_{j=1}^{n}x_j \gamma_j)})=\sum_{l}\sum_{j=1}^{n}\overline{b_{il}}\otimes b'_{il}x_j \gamma_j=\sum_{j=1}^{n}(\sum_{l}\overline{b_{il}}\otimes b'_{il}x_j)\gamma_j$. Here we use the fact that the elements of $B$ commute with the elements of $R$ under multiplication. Since $\delta_{q}^{i}:B\rightarrow B\otimes_{k}B$ is a $B^e$-homomorphism, $\sum_{l}x_j b_{il}\otimes b'_{il}=x_j(\sum_{l} b_{il}\otimes b'_{il})=x_j\delta_{q}^{i}(1)=\delta_{q}^{i}(x_j)=\delta_{q}^{i}(1)x_j=(\sum_{l} b_{il}\otimes b'_{il})x_j=\sum_{l} b_{il}\otimes b'_{il}x_j$ in $B\otimes_{k}B$. Since $\sum_{l}\overline{x_j b_{il}}\otimes b'_{il}\in T\otimes_{R}A$ (resp. $\sum_{l} \overline{b_{il}}\otimes b'_{il}x_j\in T\otimes_{R}A$) is the image of $\sum_{l}x_j b_{il}\otimes b'_{il}$ (resp. $\sum_{l} b_{il}\otimes b'_{il}x_j$) under the composition of morphisms $B\otimes_{k}B\rightarrow A\otimes_{k}A\rightarrow A\otimes_{R}A\rightarrow T\otimes_{R}A$, $\sum_{l}\overline{x_j b_{il}}\otimes b'_{il}=\sum_{l} \overline{b_{il}}\otimes b'_{il}x_j$ in $T\otimes_{R}A$. Therefore $(h\otimes 1)(\nu\otimes 1)(\delta_{q}^{i}\otimes 1)(r)=\sum_{j=1}^{n}(\sum_{l}\overline{x_jb_{il}}\otimes b'_{il})\gamma_j=\sum_{j=1}^{n}(\sum_{l}\overline{b_{il}}\otimes b'_{il}x_j)\gamma_j=(\phi\otimes 1)(\delta_{q}^{i}\otimes 1)h(r)$ and the diagram in Figure $6$ commutes.

\medskip
{\it Step 1.3: To show that $-\otimes_{A}M_q$ induces bijections between $\underline{\mathrm{Hom}}_{A}(X,Y[i])$ and $\underline{\mathrm{Hom}}_{A}(X\otimes_{A}M_q,Y[i]\otimes_{A}M_q)$ for $X$, $Y\in T^{\perp}$ and $i=0$, $1$.}

For each $X\in\underline{\mathrm{mod}}$-$A$, $\underline{\mathrm{Hom}}_{A}(T,X)=\underline{\mathrm{Hom}}_{A}(F(R/radR),X)\cong\underline{\mathrm{Hom}}_{R}(R/radR,X_R)$. Since $R$ is symmetric, $T^{\perp}=\{X\in\underline{\mathrm{mod}}$-$A\mid X_R$ projective$\}$. Since $A_R$ is projective, $T^{\perp}$ is closed under $[n]=\Omega_{A}^{-n}:\underline{\mathrm{mod}}$-$A\rightarrow\underline{\mathrm{mod}}$-$A$ for all $n\in\mathbb{Z}$. Therefore it is suffice to show that $-\otimes_{A}M_q$ is fully faithful when is restricted to $T^{\perp}$. Since there exists a triangle $\Omega_{A^{e}}(A)\xrightarrow{\underline{w_1}} M_1\xrightarrow{\underline{i_1}}(A\otimes_{R}A)^{m_0}\xrightarrow{\underline{d_0}}A$ in $\underline{\mathrm{lrp}}(A)$, and since $X\otimes_{A}(A\otimes_{R}A)^{m_0}=0$ in $\underline{\mathrm{mod}}$-$A$ for $X\in T^{\perp}$, $\underline{w_1}$ induces a natural isomorphism between functors $-\otimes_{A}\Omega_{A^{e}}(A):T^{\perp}\rightarrow\underline{\mathrm{mod}}$-$A$ and $-\otimes_{A}M_1:T^{\perp}\rightarrow\underline{\mathrm{mod}}$-$A$. Similarly, the functors $-\otimes_{A}(M_i[-1]):T^{\perp}\rightarrow\underline{\mathrm{mod}}$-$A$ and $-\otimes_{A}M_{i+1}:T^{\perp}\rightarrow\underline{\mathrm{mod}}$-$A$ are natural isomorphic for $1\leq i\leq q-1$. Therefore $-\otimes_{A}M_q:T^{\perp}\rightarrow\underline{\mathrm{mod}}$-$A$ is natural isomorphic to $\Omega_{A}^{q}(-)\cong-\otimes_{A}\Omega_{A^{e}}^{q}(A):T^{\perp}\rightarrow\underline{\mathrm{mod}}$-$A$, which implies that $-\otimes_{A}M_q$ is fully faithful when is restricted to $T^{\perp}$.

\medskip
{\it Step 1.4: To show that $-\otimes_{A}M_q$ induces bijections between $\underline{\mathrm{Hom}}_{A}(T,X[i])$ (resp. $\underline{\mathrm{Hom}}_{A}(X,T[i])$) and $\underline{\mathrm{Hom}}_{A}(T\otimes_{A}M_q,X[i]\otimes_{A}M_q)$ (resp. $\underline{\mathrm{Hom}}_{A}(X\otimes_{A}M_q,T[i]\otimes_{A}M_q)$) for $X\in T^{\perp}$ and for $i=0$, $1$.}

For each $X\in\underline{\mathrm{mod}}$-$A$, we have $$\underline{\mathrm{Hom}}_{A}(X,T)=\underline{\mathrm{Hom}}_{A}(X,F(R/radR))\cong\underline{\mathrm{Hom}}_{R}(X_R,R/radR).$$ Therefore $\prescript{\perp}{}{T}=\{X\in\underline{\mathrm{mod}}$-$A\mid X_R$ is projective$\}=T^{\perp}$. Since $T^{\perp}=\prescript{\perp}{}{T}$ is closed under $[n]=\Omega_{A}^{-n}:\underline{\mathrm{mod}}$-$A\rightarrow\underline{\mathrm{mod}}$-$A$ for all $n\in\mathbb{Z}$, $\underline{\mathrm{Hom}}_{A}(T,X[i])=0$ and $\underline{\mathrm{Hom}}_{A}(X,T[i])=0$ for $X\in T^{\perp}$ and for $i=0$, $1$. Since $T\otimes_{A}M_q\cong T$ in $\underline{\mathrm{mod}}$-$A$ and $Y\otimes_{A}M_q\cong Y[-q]$ in $\underline{\mathrm{mod}}$-$A$ for every $Y\in T^{\perp}$, $\underline{\mathrm{Hom}}_{A}(T\otimes_{A}M_q,X[i]\otimes_{A}M_q)=0$ and $\underline{\mathrm{Hom}}_{A}(X\otimes_{A}M_q,T[i]\otimes_{A}M_q)=0$ for $X\in T^{\perp}$ and for $i=0$, $1$.

By Step 1.1 $\sim$ Step 1.4, we have shown that $-\otimes_{A}M_q:\underline{\mathrm{mod}}$-$A\rightarrow\underline{\mathrm{mod}}$-$A$ is a stable auto-equivalence of $A$ when $A$ is indecomposable.

\medskip
{\it Case 2: Assume that $A$ is decomposable.}

Let $A=A_1\times\cdots\times A_p\times A_{p+1}\times\cdots\times A_r$ be the decomposition of $A$ into indecomposable blocks, where $A_{p+1}$, $\cdots$, $A_r$ are all semisimple blocks of $A$. Let $T_A=(R/radR)\otimes_{A}A\cong A/(radR)A$ and suppose $A_{1}$, $\cdots$, $A_m$ ($m\leq p$) be all indecomposable blocks of $A$ such that there exists an indecomposable non-projective summand of $T_A$ which belongs to the block. Then $\underline{\mathrm{mod}}$-$A_i$ is contained in $T^{\perp}$ for each $m+1\leq i\leq p$. Let $\mathscr{C}=\{T\}\cup T^{\perp}$ be a strong spanning class of $\underline{\mathrm{mod}}$-$A$.

Similar to Case 1, the following statements are still true: \\ $(i)$ $T^{\perp}=\prescript{\perp}{}{T}$ is closed under $[n]=\Omega_{A}^{-n}:\underline{\mathrm{mod}}$-$A\rightarrow\underline{\mathrm{mod}}$-$A$ for all $n\in\mathbb{Z}$; \\ $(ii)$ $T\otimes_{A}M_q\cong T$ in $\underline{\mathrm{mod}}$-$A$ and $X\otimes_{A}M_q\cong X[-q]$ in $\underline{\mathrm{mod}}$-$A$ for every $X\in T^{\perp}$; \\ $(iii)$ $-\otimes_{A}M_q$ induces bijections between $\underline{\mathrm{Hom}}_{A}(X,Y[i])$ and $\underline{\mathrm{Hom}}_{A}(X\otimes_{A}M_q,(Y[i])\otimes_{A}M_q)$ for all $X$, $Y\in\mathscr{C}$ and for all $i=0$, $1$.

Since the functor $-\otimes_{A}M_q:\underline{\mathrm{mod}}$-$A\rightarrow\underline{\mathrm{mod}}$-$A$ has both left and right adjoints, by statement $(iii)$ and Proposition \ref{fully faithful} it is fully faithful.

Let $T\cong \oplus_{i=1}^{m} T_i$ in $\underline{\mathrm{mod}}$-$A$, where $T_i\in\underline{\mathrm{mod}}$-$A_i$. Then $T_i\neq 0$ in $\underline{\mathrm{mod}}$-$A_i$ for each $1\leq i\leq m$. Since the functor $-\otimes_{A}M_q:\underline{\mathrm{mod}}$-$A\rightarrow\underline{\mathrm{mod}}$-$A$ is fully faithful and since $\underline{\mathrm{mod}}$-$A_i$ is an indecomposable triangulated category for $1\leq i\leq p$, by Lemma \ref{block}, for each $1\leq i\leq m$, $T_i\otimes_{A}M_q\in\underline{\mathrm{mod}}$-$A_{\sigma(i)}$ for some $1\leq \sigma(i)\leq p$. Since $T\otimes_{A}M_q\cong T$ in $\underline{\mathrm{mod}}$-$A$, we implies that $\sigma$ is a permutation of $\{1$, $\cdots$, $m\}$ and $T_i\otimes_{A}M_q\cong T_{\sigma(i)}$ for each $1\leq i\leq m$. By Lemma \ref{block}, $-\otimes_{A}M_q$ induces functors $\underline{\mathrm{mod}}$-$A_i\rightarrow\underline{\mathrm{mod}}$-$A_{\sigma(i)}$ for each $1\leq i\leq m$. Since $X\otimes_{A}M_q\cong X[-q]$ in $\underline{\mathrm{mod}}$-$A$ for every $X\in T^{\perp}$ and since $\underline{\mathrm{mod}}$-$A_i$ is contained in $T^{\perp}$ for each $m+1\leq i\leq p$, $-\otimes_{A}M_q$ induces functors $\underline{\mathrm{mod}}$-$A_i\rightarrow\underline{\mathrm{mod}}$-$A_{i}$ for each $m+1\leq i\leq p$.

Let $\tau$ be a permutation of $\{1$, $\cdots$, $p\}$ such that $\tau(i)=\sigma(i)$ for $1\leq i\leq m$ and $\tau(i)=i$ for $m+1\leq i\leq p$. Since $-\otimes_{A}M_q$ induces functors $\underline{\mathrm{mod}}$-$A_i\rightarrow\underline{\mathrm{mod}}$-$A_{\tau(i)}$ for each $1\leq i\leq p$, to show $-\otimes_{A}M_q:\underline{\mathrm{mod}}$-$A\rightarrow\underline{\mathrm{mod}}$-$A$ is a triangulated equivalence, it suffices to show each $-\otimes_{A}M_q:\underline{\mathrm{mod}}$-$A_i\rightarrow\underline{\mathrm{mod}}$-$A_{\tau(i)}$ is a triangulated equivalence for each $1\leq i\leq p$.

Let $1=\sum_{i=1}^{r} e_i$, where $e_i\in A_i$. For each $1\leq i\leq p$, $-\otimes_{A}M_q$ is natural isomorphic to $-\otimes_{A}e_i M_q$ as functors from $\underline{\mathrm{mod}}$-$A_i$ to $\underline{\mathrm{mod}}$-$A_{\tau(i)}$. For each $X\in\underline{\mathrm{mod}}$-$A_i$, $X\otimes_{A}e_i M_q\cong \oplus_{j=1}^{p}(X\otimes_{A}e_i M_q e_j)$ in $\underline{\mathrm{mod}}$-$A$. Since $X\otimes_{A}e_i M_q\in\underline{\mathrm{mod}}$-$A_{\tau(i)}$, $X\otimes_{A}e_i M_q e_j=0$ in $\underline{\mathrm{mod}}$-$A_j$ for $j\neq \tau(i)$. Then $-\otimes_{A}M_q$ is natural isomorphic to $-\otimes_{A}e_i M_q e_{\tau(i)}$ as functors from $\underline{\mathrm{mod}}$-$A_i$ to $\underline{\mathrm{mod}}$-$A_{\tau(i)}$ for each $1\leq i\leq p$. Since $e_i M_q e_{\tau(i)}$ is a summand of $e_i M_i$ as left $A_i$-module, and since $e_i M_i$ is projective as a left $A_i$-module, so is $e_i M_q e_{\tau(i)}$. Similarly, $e_i M_q e_{\tau(i)}$ is also projective as a right $A_{\tau(i)}$-module. Therefore $e_i M_q e_{\tau(i)}$ is a left-right projective $A_i$-$A_{\tau(i)}$-bimodule. Since both $A_i$ and $A_{\tau(i)}$ are symmetric, $-\otimes_{A}D(e_i M_q e_{\tau(i)}):\underline{\mathrm{mod}}$-$A_{\tau(i)}\rightarrow\underline{\mathrm{mod}}$-$A_i$ is both the left and the right adjoint of $-\otimes_{A}e_i M_q e_{\tau(i)}:\underline{\mathrm{mod}}$-$A_i\rightarrow\underline{\mathrm{mod}}$-$A_{\tau(i)}$. Since $-\otimes_{A}e_i M_q e_{\tau(i)}:\underline{\mathrm{mod}}$-$A_i\rightarrow\underline{\mathrm{mod}}$-$A_{\tau(i)}$ is fully faithful, $\underline{\mathrm{mod}}$-$A_i$ is nonzero, and $\underline{\mathrm{mod}}$-$A_{\tau(i)}$ is indecomposable as a triangulated category, it follows from Proposition \ref{equivalence} that $-\otimes_{A}e_i M_q e_{\tau(i)}:\underline{\mathrm{mod}}$-$A_i\rightarrow\underline{\mathrm{mod}}$-$A_{\tau(i)}$ is a triangulated equivalence. Therefore $-\otimes_{A}M_q:\underline{\mathrm{mod}}$-$A_i\rightarrow\underline{\mathrm{mod}}$-$A_{\tau(i)}$ is a triangulated equivalence.
\end{proof}

\section{A variation of the construction in previous section}

There exist some examples of stable equivalences (cf. Subsection 6.1) which do not satisfies Assumptions 1 in last section, however if we modify some conditions, we may obtain a similar proposition, which will include these examples.

In this section, we make the following

{\bf Assumption 2:} Let $k$ be a field, $A$ be a symmetric $k$-algebra, $R$ be a non-semisimple symmetric subalgebra of $A$ such that $A_R$ is projective. Let $B$ be another subalgebra of $A$, such that the following conditions hold: \\ $(a')$ $(radR)B=B(radR)$; \\ $(b)$ $B\otimes_{k} (R/radR)\xrightarrow{\phi} A/(radR)A$, $b\otimes\overline{1}\mapsto\overline{b}$ is an isomorphism in $\underline{\mathrm{mod}}$-$R$;\\ $(c)$ $B$ has a periodic free $B^e$-resolution

$0\rightarrow B\xrightarrow{\delta_q} (B\otimes_{k}B)^{m_{q-1}}\xrightarrow{\delta_{q-1}}\cdots\rightarrow (B\otimes_{k}B)^{m_{1}}\xrightarrow{\delta_1} (B\otimes_{k}B)^{m_{0}}\xrightarrow{\delta_0} B\rightarrow 0;$ \\
$(d)$ The image $x$ of $\delta_q(1)$ in $(A\otimes_{R}A)^{m_{q-1}}$ satisfies $rx=xr$ for all $r\in R$; \\ $(e)$ There exists a complex

$(A\otimes_{R}A)^{m_{q-1}}\xrightarrow{d_{q-1}}(A\otimes_{R}A)^{m_{q-2}}\xrightarrow{d_{q-2}}\cdots\rightarrow (A\otimes_{R}A)^{m_{1}}\xrightarrow{d_1} (A\otimes_{R}A)^{m_{0}}\xrightarrow{d_0} A\rightarrow 0;$ \\
of $A^e$-modules such that the diagram
$$\xymatrix@R=3pc@C=2pc@W=0mm{
(B\otimes_{k}B)^{m_{q-1}}\ar[r]^{\delta_{q-1}}\ar[d]& (B\otimes_{k}B)^{m_{q-2}}\ar[r]^(0.7){\delta_{q-2}}\ar[d]&\cdots\ar[r] &(B\otimes_{k}B)^{m_{1}}\ar[r]^{\delta_1}\ar[d] & (B\otimes_{k}B)^{m_{0}}\ar[r]^(0.7){\delta_0}\ar[d] & B\ar[r]\ar[d] & 0 \\
(A\otimes_{R}A)^{m_{q-1}}\ar[r]^{d_{q-1}}& (A\otimes_{R}A)^{m_{q-2}}\ar[r]^(0.7){d_{q-2}}&\cdots\ar[r] &(A\otimes_{R}A)^{m_{1}}\ar[r]^{d_1} & (A\otimes_{R}A)^{m_{0}}\ar[r]^(0.7){d_0} & A\ar[r] & 0 \\
}$$
is commutative, where the vertical morphisms are the obvious morphisms.

Note that the condition $(a')$ is a generalization of $(a)$ in Assumption 1, the conditions $(b)$ and $(c)$ are the same as in Assumption 1, and the conditions $(d)$ and $(e)$ are new. Clearly, if the triple $(A,R,B)$ satisfies Assumption 1, then it also satisfies Assumption 2.

Similar to Lemma \ref{from complex to triangles}, there exist triangles $M_1 \xrightarrow{\underline{i_1}}(A\otimes_{R}A)^{m_{0}}\xrightarrow{\underline{d_{0}}} A\rightarrow$, $M_2 \xrightarrow{\underline{i_2}}(A\otimes_{R}A)^{m_{1}}\xrightarrow{\underline{f_{1}}} M_1 \rightarrow$, $\cdots$, $M_q \xrightarrow{\underline{i_q}}(A\otimes_{R}A)^{m_{q-1}}\xrightarrow{\underline{f_{q-1}}} M_{q-1} \rightarrow$ of $\underline{\mathrm{lrp}}(A)$ such that $i_p f_p =d_p$ for $1\leq p\leq q-1$. We have following proposition, which is an analogy of Theorem \ref{auto-equivalence}.

\begin{Thm}\label{auto-equivalence 2}
Let $(A,R,B)$ be the triple that satisfies Assumption 2. If $M_q$ is the $A$-$A$-bimodule defined above, then $-\otimes_{A}M_q:\underline{\mathrm{mod}}$-$A\rightarrow\underline{\mathrm{mod}}$-$A$ is a stable auto-equivalence of $A$.
\end{Thm}

\begin{proof}
Since $(radR)B=B(radR)$, $T_A=A/(radR)A$ and $\Omega_{A}T=(radR)A$ becomes $B$-$A$-bimodules. The proof is similar to the proof of Theorem \ref{auto-equivalence}. The only difficulty is to show the diagrams in Figure $3$ and Figure $6$ are commutative.

\medskip
{\it To show that the diagrams in Figure $3$ are commutative.}

Since the image $x$ of $\delta_q(1)$ in $(A\otimes_{R}A)^{m_{q-1}}$ satisfies $rx=xr$ for all $r\in R$, we have $\sum_{l} rb_{il}\otimes b'_{il}=\sum_{l} b_{il}\otimes b'_{il}r$ in $A\otimes_{R}A$ for all $r\in R$. Therefore $\sum_{l} \overline{rb_{il}}\otimes b'_{il}=\sum_{l} \overline{b_{il}}\otimes b'_{il}r$ in $T\otimes_{R}A$ for all $r\in R$. Moreover, since $\delta_{q}^{i}$ is a $B^e$-homomorphism, $\sum_{l} bb_{il}\otimes b'_{il}=\sum_{l} b_{il}\otimes b'_{il}b$ in $B\otimes_{k}B$ for all $b\in B$, and therefore $\sum_{l} \overline{bb_{il}}\otimes b'_{il}=\sum_{l} \overline{b_{il}}\otimes b'_{il}b$ in $T\otimes_{R}A$ for all $b\in B$. We have $(h\otimes 1)(\phi\otimes 1)(\delta_{q}^{i}\otimes 1)(\overline{1})=\sum_{l,j}\overline{\beta_j r_j b_{il}}\otimes b'_{il}=\sum_{j}\beta_j\cdot(\sum_{l} \overline{r_j b_{il}}\otimes b'_{il})=\sum_{j}\beta_j\cdot(\sum_{l} \overline{b_{il}}\otimes b'_{il} r_j)=\sum_{j}\sum_{l} (\overline{\beta_j b_{il}}\otimes b'_{il})\cdot r_j=\sum_{j}\sum_{l} (\overline{b_{il}}\otimes b'_{il}\beta_j)\cdot r_j=\sum_{j}\sum_{l} \overline{b_{il}}\otimes b'_{il}\beta_j r_j=(\phi\otimes 1)(\delta_{q}^{i}\otimes 1)h(\overline{1})$ and the diagram in Figure $3$ commutes.

\medskip
{\it To show that the diagrams in Figure $6$ are commutative.}

For $r\in radR\subseteq (radR)A=\Omega_{A}T$, $(\delta_{q}^{i}\otimes 1)(r)=\sum_{l} b_{il}\otimes b'_{il}r$. There is a commutative diagram
$$\xymatrix@R=3pc@C=2pc@W=0mm{
B\otimes_{k}(radR)A\ar[r]^(0.6){u}\ar[d]^{\nu\otimes 1} & A\otimes_{k}A\ar[d]^{p} \\
(radR)A\otimes_{R}A\ar[r]^(0.6){v} & A\otimes_{R}A \\
}$$
in $\mathrm{mod}$-$A$, where $u$, $v$, $p$ are the obvious morphisms. Since $\sum_{l} rb_{il}\otimes b'_{il}=\sum_{l} b_{il}\otimes b'_{il}r$ in $A\otimes_{R}A$, $(pu)(\sum_{l} b_{il}\otimes b'_{il}r)=\sum_{l} rb_{il}\otimes b'_{il}=v(\sum_{l} rb_{il}\otimes b'_{il})$. Since $v$ is injective and $pu=v(\nu\otimes 1)$, $(\nu\otimes 1)(\sum_{l} b_{il}\otimes b'_{il}r)=\sum_{l} rb_{il}\otimes b'_{il}$. Then $(h\otimes 1)(\nu\otimes 1)(\delta_{q}^{i}\otimes 1)(r)=(h\otimes 1)(\sum_{l} rb_{il}\otimes b'_{il})=\sum_{l,j}\overline{x_j \gamma_j b_{il}}\otimes b'_{il}=\sum_{j}x_j\cdot(\sum_{l}\overline{\gamma_j b_{il}}\otimes b'_{il})=\sum_{j}x_j\cdot(\sum_{l}\overline{b_{il}}\otimes b'_{il}\gamma_j)=\sum_{j}(\sum_{l}\overline{x_j b_{il}}\otimes b'_{il})\cdot\gamma_j=\sum_{j}(\sum_{l}\overline{b_{il}}\otimes b'_{il}x_j)\cdot\gamma_j=\sum_{l,j}\overline{b_{il}}\otimes b'_{il}x_j\gamma_j=(\phi\otimes 1)(\delta_{q}^{i}\otimes 1)h(r)$. So the diagram in Figure $6$ is commutative.
\end{proof}

Recall that an $A$-module $X$ is called a relatively $R$-projective module if $X$ is isomorphic to a direct summand of $X\otimes_RA_A$. For $A$-modules $X$, $Y$ with $Y$ relatively $R$-projective, an $A$-homomorphism $f:Y\rightarrow X$ is called a relatively $R$-projective cover of $X$ if any $A$-homomorphism $g:Z\rightarrow X$ with $Z$ relatively $R$-projective factors through $f$. This is equivalent to the fact that $f$ is a split epimorphism as an $R$-homomorphism.

\begin{Prop} {\rm(Compare to \cite[Proposition 6.5]{Dugas2016})} \label{relative syzygy}
Let $\rho=-\otimes_{A}M_q:\underline{\mathrm{mod}}$-$A\rightarrow\underline{\mathrm{mod}}$-$A$ be the stable auto-equivalence of $A$ in Theorem \ref{auto-equivalence 2}. If both $A$, $R$, $B$ are elementary local $k$-algebras, then $\rho(k)$ is isomorphic to $\Omega_{R}^{q}(k)$ up to a summand of a relatively $R$-projective module. (Note that $\Omega_{R}(X)$ denotes the kernel of some relatively $R$-projective cover of $_AX$ and it is determined up to a summand of a relatively $R$-projective module.)
\end{Prop}

\begin{proof}
Since $R/radR=k$, we have an isomorphism $\phi:B\rightarrow k\otimes_{R}A$, $b\mapsto 1\otimes b$ in $\underline{\mathrm{mod}}$-$R$, where the $R$-module structure of $B$ is induced from the epimorphism $R\rightarrow k$. Applies the functors $k\otimes_{B}-$ and $k\otimes_{A}-$ to the complex $0\rightarrow B\xrightarrow{\delta_q} (B\otimes_{k}B)^{m_{q-1}}\xrightarrow{\delta_{q-1}}\cdots\rightarrow (B\otimes_{k}B)^{m_{1}}\xrightarrow{\delta_1} (B\otimes_{k}B)^{m_{0}}\xrightarrow{\delta_0} B\rightarrow 0$ and the complex $(A\otimes_{R}A)^{m_{q-1}}\xrightarrow{d_{q-1}}\cdots\rightarrow (A\otimes_{R}A)^{m_{1}}\xrightarrow{d_{1}} (A\otimes_{R}A)^{m_{0}}\xrightarrow{d_{0}} A$ respectively, we get a commutative diagram in $\mathrm{mod}$-$R$:

$$\xymatrix@R=2pc@C=2pc@W=0mm{
 0 \ar[r] & k \ar[r]^(0.3){1\otimes \delta_{q}} & B^{m_{q-1}}\ar[r]^(0.7){1\otimes \delta_{q-1}}\ar[d]_{\phi^{m_{q-1}}} & \cdots\ar[r] & B^{m_{1}}\ar[r]^{1\otimes \delta_{1}}\ar[d]_{\phi^{m_{1}}} & B^{m_{0}}\ar[r]^(0.6){1\otimes \delta_{0}}\ar[d]_{\phi^{m_{0}}} & k\ar[r]\ar@{=}[d] & 0 \\
  & & (k\otimes_{R}A)^{m_{q-1}}\ar[r]^(0.7){1\otimes d_{q-1}} & \cdots\ar[r] & (k\otimes_{R}A)^{m_{1}}\ar[r]^{1\otimes d_{1}} & (k\otimes_{R}A)^{m_{0}}\ar[r]^(0.7){1\otimes d_{0}} & k &  \\
 }.$$
Since the first row of the diagram is split exact as a complex of $k$-modules, it is also split exact as a complex of $R$-modules. Similar to the argument in Step 1.1 of the proof of Theorem \ref{auto-equivalence}, we have isomorphisms of split triangles
$$\xymatrix@R=2pc@C=2pc@W=0mm{
 L_{l+1}\ar[r]\ar[d] & B^{m_{l}}\ar[r]\ar[d]_{\underline{\phi^{m_{l}}}} &L_l\ar[r]\ar[d] & \\
 k\otimes_{A}M_{l+1} \ar[r]^{\underline{1\otimes i_{l+1}}} & (k\otimes_{R}A)^{m_{l}}\ar[r]^{\underline{1\otimes f_{l}}} & k\otimes_{A}M_l\ar[r] & \\
 }$$
in $\underline{\mathrm{mod}}$-$R$ for $0\leq l\leq q-1$, where $L_0=L_q=k$, $M_0=A$, $f_0=d_0$. Therefore $\underline{1\otimes f_{l}}:(k\otimes_{R}A)^{m_{l}}\rightarrow k\otimes_{A}M_l$ are split epimorphisms in $\underline{\mathrm{mod}}$-$R$ for $0\leq l\leq q-1$.

For every $0\leq l\leq q-1$ and for every $R$-module $X_R$, we have a commutative diagram
$$\xymatrix@R=2pc@C=6pc@W=0mm{
 \underline{\mathrm{Hom}}_{A}(FX,(k\otimes_{R}A)^{m_{l}})\ar[d]\ar[r]^{\underline{\mathrm{Hom}}_{A}(FX,\underline{1\otimes f_{l}})}& \underline{\mathrm{Hom}}_{A}(FX,k\otimes_{A}M_l)\ar[d] \\
 \underline{\mathrm{Hom}}_{R}(X,(k\otimes_{R}A)_{R}^{m_{l}})\ar[r]^{\underline{\mathrm{Hom}}_{R}(X,\underline{1\otimes f_{l}})}& \underline{\mathrm{Hom}}_{R}(X,(k\otimes_{A}M_l)_{R}) \\
 },$$
where the vertical arrows are isomorphisms. Since $\underline{1\otimes f_{l}}:(k\otimes_{R}A)^{m_{l}}\rightarrow k\otimes_{A}M_l$ is a split epimorphism in $\underline{\mathrm{mod}}$-$R$, $\underline{\mathrm{Hom}}_{R}(X,(k\otimes_{R}A)_{R}^{m_{l}})\rightarrow \underline{\mathrm{Hom}}_{R}(X,(k\otimes_{A}M_l)_{R})$ is surjective, therefore $\underline{\mathrm{Hom}}_{A}(FX,(k\otimes_{R}A)^{m_{l}})\rightarrow \underline{\mathrm{Hom}}_{A}(FX,k\otimes_{A}M_l)$ is surjective. Then the morphism $\underline{1\otimes f_{l}}:(k\otimes_{R}A)^{m_{l}}\rightarrow k\otimes_{A}M_l$ is a right $F(\underline{\mathrm{mod}}$-$R)$-approximation. It follows that the $A$-homomorphism $(1\otimes f_{l},\pi_l):(k\otimes_{R}A)^{m_{l}}\oplus P_l\rightarrow k\otimes_{A}M_l$ is a relatively $R$-projective cover of $k\otimes_{A}M_l$, where $\pi_l:P_l\rightarrow k\otimes_{A}M_l$ is the projective cover of $k\otimes_{A}M_l$. By the triangle $k\otimes_{A}M_{l+1} \xrightarrow{\underline{1\otimes i_{l+1}}}  (k\otimes_{R}A)^{m_{l}}\xrightarrow{\underline{1\otimes f_{l}}} k\otimes_{A}M_l\rightarrow$ in $\underline{\mathrm{mod}}$-$A$, we see that $k\otimes_{A}M_{l+1}\cong\Omega_{R}(k\otimes_{A}M_l)$. Therefore $\rho(k)=k\otimes_{A}M_{q}\cong\Omega_{R}(k\otimes_{A}M_{q-1})\cong\cdots\cong \Omega_{R}^{q}(k)$.
\end{proof}

\begin{Rem1}\label{remark}
Since the stable auto-equivalence in Theorem \ref{auto-equivalence} is a special case of the stable auto-equivalence in Theorem \ref{auto-equivalence 2}, it also satisfies Proposition \ref{relative syzygy}.
\end{Rem1}

\section{Endo-trivial modules over finite p-groups}

Let $k$ be a field of characteristic $p$ with $p$ prime, $P$ be a finite $p$-group and $kP$ be its group algebra. A $kP$-module $M$ is called endo-trivial if $\mathrm{End}_{k}(M)\cong k\oplus P$ for some projective module $P$. Two endo-trivial modules $M$, $N$ are said to be equivalent if $M\oplus P\cong N\oplus Q$ for some projective modules $P$, $Q$. The group $T(P)$ has elements consisting of equivalence classes $[M]$ of endo-trivial modules $M$. The operation is given by $[M]+[N]=[M\otimes_{k}N]$, see \cite[Section 3]{CT2004}.

Note that the stable auto-equivalences of Morita type of $kP$ are precisely induced by endo-trivial modules. The next proposition shows that in most cases, our construction recovers all the stable auto-equivalences of $kP$ corresponding to endo-trivial modules.

Let $A=kP$ and $R=kS$, $B=kL$ for some subgroups $S$, $L$ of $P$. Suppose that the triple $(A,R,B)$ satisfies Assumption 1 of Section 3. Let $\rho_{S,L}:=-\otimes_{A}M_q:\underline{\mathrm{mod}}$-$A\rightarrow\underline{\mathrm{mod}}$-$A$ be the stable auto-equivalence of $A$ in Theorem \ref{auto-equivalence}. Since $\underline{\mathrm{End}}_{A}(\rho_{S,L}(k))\cong\underline{\mathrm{End}}_{A}(k)\cong k$, by \cite[Theorem 1]{C1998}, $\rho_{S,L}(k)$ is an endo-trivial module.

\begin{Prop}\label{generator of T(P)}
Let $P$ be a finite $p$-group which is not generalized quaternion. Then there exist finitely many pairs ($S_i$, $L_i$) of subgroups of $P$ such that the following conditions hold: \\
(1) Each pair ($S_i$, $L_i$) gives a triple $(A,kS_i,kL_i)$ satisfying Assumption 1; \\
(2) $T(P)$ is generated by $[\Omega_{kP}(k)]$ and elements of the form $[\rho_{S_i,L_i}(k)]$, where $\rho_{S_i,L_i}$ is the stable auto-equivalence of $A=kP$ defined as above.
\end{Prop}

In the following, for a subgroup $H$ of a group $G$, we denote by $N_{G}(H)$ and $C_{G}(H)$ the normalizer and the centralizer of $H$ in $G$ respectively.

\begin{Lem}\label{double coset}
Let $G$ be a group, $H$ be a subgroup of $G$ of order $p$ with $p$ prime. Then for every $g\in G$,
\begin{equation}
|HgH|= \begin{cases}
p, & \text{if } g\in N_{G}(H); \\
p^2, & otherwise.
\end{cases}
\end{equation}
\end{Lem}

\begin{proof}
If $g\notin N_{G}(H)$, then $g^{-1}Hg\neq H$. Since $|g^{-1}Hg|=|H|=p$, we have $|g^{-1}Hg\cap H|=1$. Therefore $|HgH|=|g^{-1}HgH|=\frac{|g^{-1}Hg||H|}{|g^{-1}Hg\cap H|}=p^2$.
\end{proof}

\begin{Lem}\label{centralizer and normalizer}
Let $P$ be a finite $p$-group and $H$ be a subgroup of $P$ order $p$, then $C_{P}(H)=N_{P}(H)$.
\end{Lem}

\begin{proof}
There is a group homomorphism $\phi:N_{P}(H)\rightarrow Aut(H)$ such that $\phi(g)(h)=ghg^{-1}$ for all $g\in N_{P}(H)$ and $h\in H$. Moreover, the kernel of $\phi$ is $C_{P}(H)$. Since $Aut(H)\cong Aut(\mathbb{Z}/p\mathbb{Z})\cong\mathbb{Z}/p\mathbb{Z}^{\times}$, $|Aut(H)|=p-1$. Therefore $[N_{P}(H):C_{P}(H)]$ divides $p-1$. Since $[N_{P}(H):C_{P}(H)]$ is a power of $p$, it must equal to $1$.
\end{proof}

\begin{Lem}\label{periodic resolution}
Let $G$ be a finite group. If the trivial $G$-module $k$ has a periodic free resolution of periodic $n$, then $kG$ has a periodic free resolution as $kG$-$kG$-bimodule of the same periodic.
\end{Lem}

\begin{proof}
For $X\in \mathrm{mod}$-$kG$, define a $kG$-$kG$-bimodule structure on $X\otimes_{k}kG$ by the formulas $g\cdot (x\otimes\mu)=x\otimes g\mu$ and $(x\otimes\mu)\cdot g=xg\otimes\mu g$. It can be shown that the map $X\mapsto X\otimes_{k}kG$ defines a functor $\Phi$ from $\mathrm{mod}$-$kG$ to $kG$-$\mathrm{mod}$-$kG$. Since the trivial $G$-module $k$ has a periodic free resolution, there exists an exact sequence $0\rightarrow k\rightarrow F_{n-1}\rightarrow\cdots\rightarrow F_1\rightarrow F_0\rightarrow k\rightarrow 0$ of $kG$-modules, where $F_0$, $\cdots$, $F_{n-1}$ are free $kG$-modules. Let $M=kG\otimes_{k}kG$ be the free $kG$-$kG$-bimodule of rank 1. Then the map $\Phi(kG)\rightarrow M$, $g\otimes h\mapsto hg^{-1}\otimes g$ is an isomorphism of $kG$-$kG$-bimodules. So $\Phi$ sends free $kG$-modules to free $kG$-$kG$-bimodules. Applies the functor $\Phi$ to the exact sequence $0\rightarrow k\rightarrow F_{n-1}\rightarrow\cdots\rightarrow F_1\rightarrow F_0\rightarrow k\rightarrow 0$, we get an exact sequence $0\rightarrow \Phi(k)\rightarrow \Phi(F_{n-1})\rightarrow\cdots\rightarrow \Phi(F_1)\rightarrow \Phi(F_0)\rightarrow \Phi(k)\rightarrow 0$ of $kG$-$kG$-bimodules with $\Phi(F_0)$, $\cdots$, $\Phi(F_{n-1})$ free. Note that $\Phi(k)\cong kG$ as $kG$-$kG$-bimodules.
\end{proof}

\begin{proof}[{\bf Proof of Proposition \ref{generator of T(P)}}]

\medskip
{\it Case 1: Assume that $P$ is a finite $p$-group having a maximal elementary abelian subgroup of rank $2$.}

\medskip
{\it Case 1.1: $P$ is not semi-dihedral.}

By \cite[Theorem 7.1]{CT2004}, $T(P)$ is a free abelian group generated by the classes of the modules $\Omega_{kP}(k)$, $N_2$, $\cdots$, $N_r$, where $r$ is the number of conjugacy classes of connected components of the poset of all elementary abelian subgroups of $P$ of rank at least 2 and the $N_i$ are defined as follows. For $2\leq i\leq r$, let $S_i$ be the subgroups of $P$ of order $p$ in \cite[Lemma 2.2(b)]{CT2004} with $C_{P}(S_i)=S_i\times L_i$, where $L_i$ either cyclic or generalized quaternion. Let $M_i=\Omega_{kP}^{-1}(k)\otimes_{k}\Omega_{P/S_i}(k)$, where $\Omega_{P/S_i}(k)$ denotes the kernel of a relatively $kS_i$-projective cover of the trivial $kP$-module $k$. Define
\begin{equation}
N_i= \begin{cases}
\Gamma(M_{i}^{\otimes 2}), & \text{if } L_i \text{ is cyclic of order }\geq 3; \\
M_i, & \text{if } p=2 \text{ and } L_i \text{ is cyclic of order } 2; \\
\Gamma(M_{i}^{\otimes 4}), & \text{if } p=2 \text{ and } L_i \text{ is generalized quaternion},
\end{cases}
\end{equation}
where $\Gamma(M)$ denotes the sum of all the indecomposable summands of $M$ having vertex $P$. Let $A=kP$ and $R_i=kS_i$, $B_i=kL_i$ for $2\leq i\leq r$. Note that $R_i/radR_i\cong k$. Since $L_i\leq C_{P}(S_i)$, we have $br=rb$ for any $b\in B_i$ and $r\in R_i$. Let $h_1$, $\cdots$, $h_q$ be a complete set of double coset representatives for $S_i$ in $P$ which not belong to $N_{P}(S_i)$. Since $P$ is a $p$-group and $S_i$ is a subgroup of $P$ of order $p$, by Lemma \ref{centralizer and normalizer}, $N_{P}(S_i)=C_{P}(S_i)$. Therefore $P$ is a disjoint union of double cosets $S_igS_i=gS_i$ with $g\in L_i$ and double cosets $S_i h_n S_i$ with $1\leq n\leq q$. By Lemma \ref{double coset}, $|S_i h_n S_i|=p^2$ for $1\leq n\leq q$, therefore the $R_i$-$R_i$-subbimodule $kS_i h_n S_i$ of $A$ is isomorphic to $R_i\otimes_{k}R_i$. We have $A/(radR_i)A\cong (R_i/radR_i)\otimes_{R_i}A=k\otimes_{R_i}A\cong \bigoplus_{g\in L_i}k\otimes_{R_i}kgS_i\oplus\bigoplus_{n=1}^{q}k\otimes_{R_i}kS_ih_nS_i\cong k^{|L_i|}\oplus R_{i}^{q}$ as $R_i$-modules. Moreover, the $R_i$-homomorphism $\phi_i:B_i\otimes_{R_i}(R_i/radR_i)\rightarrow A/(radR_i)A$, $b\otimes 1\mapsto \overline{b}$ is isomorphic to the inclusion morphism $k^{|L_i|}\rightarrow k^{|L_i|}\oplus R_{i}^{q}$. Therefore $\phi_i$ is an isomorphism in $\underline{\mathrm{mod}}$-$R_i$.

Let $k$ denotes the trivial $L_i$-module. When $L_i$ is cyclic, then $\Omega_{kL_i}^{2}(k)\cong k$. Moreover, when $L_i$ is cyclic of order $2$, then $\Omega_{kL_i}(k)\cong k$. When $L_i$ is generalized quaternion, by \cite[Proposition 3.16]{D1972}, $\Omega_{kL_i}^{4}(k)\cong k$. Since $B_i=kL_i$ is local, the periodic projective resolution of $k$ is also a periodic free resolution. By Lemma \ref{periodic resolution}, $B_i$ has a periodic free resolution as a $B_i$-$B_i$-bimodule of periodic $n_i$, where
\begin{equation}
n_i= \begin{cases}
2, & \text{if } L_i \text{ is cyclic of order }\geq 3; \\
1, & \text{if } p=2 \text{ and } L_i \text{ is cyclic of order } 2; \\
4, & \text{if } p=2 \text{ and } L_i \text{ is generalized quaternion}.
\end{cases}
\end{equation}
Therefore the triple $(A,R_i,B_i)$ satisfies Assumption 1 in Section $3$. By Proposition \ref{relative syzygy} and Remark \ref{remark}, $\rho_{S_i,L_i}(k)\cong\Omega_{P/S_i}^{n_i}(k)$. Since $\Omega_{P/S_i}(k)^{\otimes n_i}\oplus V\cong \Omega_{P/S_i}^{n_i}(k)\oplus W$ for some relatively $kS_i$-projective modules $V$, $W$,
\begin{equation}
N_i= \begin{cases}
\Gamma(\Omega_{kP}^{-n_i}(k)\otimes_{k}\rho_{S_i,L_i}(k)), & \text{if } L_i \text{ is cyclic of order }\geq 3, \\
& \text{ or } p=2 \text{ and } L_i \text{ is generalized quaternion}; \\
\Omega_{kP}^{-n_i}(k)\otimes_{k}\rho_{S_i,L_i}(k), & \text{if } p=2 \text{ and } L_i \text{ is cyclic of order } 2.
\end{cases}
\end{equation}
When $L_i$ is cyclic of order $\geq 3$, or when $p=2$ and $L_i$ is generalized quaternion, since both $\Omega_{kP}^{-n_i}(k)\otimes_{k}\rho_{S_i,L_i}(k)$ and $\Gamma(\Omega_{kP}^{-n_i}(k)\otimes_{k}\rho_{S_i,L_i}(k))$ are endo-trivial modules, $\Omega_{kP}^{-n_i}(k)\otimes_{k}\rho_{S_i,L_i}(k)\cong\Gamma(\Omega_{kP}^{-n_i}(k)\otimes_{k}\rho_{S_i,L_i}(k))\oplus V$ for some projective $kP$-module $V$. Therefore $[\Omega_{kP}^{-n_i}(k)\otimes_{k}\rho_{S_i,L_i}(k)]=[\Gamma(\Omega_{kP}^{-n_i}(k)\otimes_{k}\rho_{S_i,L_i}(k))]$ in $T(P)$. So $T(P)$ is generated by $[\Omega_{kP}(k)]$ and $[\rho_{S_i,L_i}(k)]$ for $2\leq i\leq r$.

\medskip
{\it Case 1.2: $P$ is semi-dihedral.}

The semi-dihedral of order $2^n$ ($n\geq 4$) is given by $SD_{2^n}=\langle x,y\mid x^{2^{n-1}}=y^2=1,yxy=x^{2^{n-2}-1}\rangle$. Let $S=\langle y\rangle$ be a subgroup of $P=SD_{2^n}$. Then $C_{P}(S)=S\times S'$, where $S'=\langle x^{2^{n-2}}\rangle$. Let $A=kP$, $R=kS$, $B=kS'$. Similar to Case 1.1, the triple $(A,R,B)$ satisfies Assumption 1. Since $B$ has a free resolution of periodic $1$ as a $B$-$B$-bimodule, by Proposition \ref{relative syzygy} and Remark \ref{remark}, $\rho_{S,S'}(k)\cong\Omega_{P/S}(k)$, which is exactly the module $L$ defined in \cite[Section 7]{CT2000}. By \cite[Theorem 7.1]{CT2000}, $T(P)$ is isomorphic to $\mathbb{Z}\oplus\mathbb{Z}/2\mathbb{Z}$, generated by $[\Omega_{kP}(k)]$ and $[\Omega_{kP}(L)]$, where the element $[\Omega_{kP}(L)]$ has order 2. Therefore $[\Omega_{kP}(k)]$ together with $[\rho_{S,L}(k)]$ generates $T(P)$.

\medskip
{\it Case 2: Assume that $P$ is a finite $p$-group which do not have a maximal elementary abelian subgroup of rank $2$.}

Since $P$ is not generalized quaternion, either $P$ is cyclic or every maximal elementary abelian subgroup of $P$ has rank at least $3$ (cf. \cite[Introduction]{CT2004}). By \cite[Corollary 8.8]{D1978} and \cite[Corollary 1.3]{CT2005}, $T(P)$ is generated by $[\Omega_{kP}(k)]$. So the conclusion also holds in this case.
\end{proof}

\begin{Rem1}
An example of $p$-group which has a maximal elementary abelian subgroup of rank $2$ and which is not semi-dihedral is the dihedral group $D_8=\langle x,y\mid x^4=y^2=1$, $yxy=x^{-1}\rangle$ of order $8$, where $E=\{1,x^2,y,x^2 y\}$ is a maximal elementary abelian subgroup of $Q_8$ of rank $2$.  An example of $p$-group whose maximal elementary abelian subgroups have rank at least $3$ is $D_8*D_8=(D_8\times D_8)/\langle (x^2,x^2)\rangle$, see \cite[Section 6]{CT2004}.
\end{Rem1}

\begin{Rem1}
For every positive integer $n\geq 2$, the generalized quaternion group $Q_{4n}$ of order $4n$ is defined by the presentation $\langle x,y\mid x^{2n}=1$, $y^2=x^{n}$, $yxy^{-1}=x^{-1}\rangle$. When $n=2$ it is the usual quaternion group. The generalized quaternion group $Q_{4n}$ is a $p$-group if and only if $n$ is a power of $2$. The reason why we exclude generalized quaternion groups in Proposition \ref{generator of T(P)} is that the endo-trivial module $L$ constructed in \cite[Section 6]{CT2000} may not be a relative syzygy of the trivial $kP$-module.
\end{Rem1}

\section{Examples in non-local case}

\subsection{}

In this subsection, let $G$ be a finite group and $N$, $H$ be subgroups of $G$ such that $N_{G}(N)=N\rtimes H$ and $|NgN|=|N|^2$ for any $g\in G-N_{G}(N)$. Let $k$ be a field whose characteristic divides $|N|$, and let $A=kG$, $R=kN$, $B=kH$. Assume that the trivial $kH$-module $k$ has a periodic free resolution.

\begin{Prop}\label{auto-equivalence given by semidirect product}
The triple $(A,R,B)$ as above satisfies Assumption 2 of Section 4, so it defines a stable auto-equivalence of $A$ by Theorem \ref{auto-equivalence 2}.
\end{Prop}

\begin{proof}
Since $N$ is a subgroup of $G$, $A_R$ is projective. We need to check that the triple $(A,R,B)$ satisfies the assumptions $(a')$ to $(e)$ at the beginning of Section 4.

Suppose the semidirect product $N\rtimes H$ is defined by the group homomorphism $\eta:H\rightarrow \mathrm{Aut}(N)$. For any $\sum_{n\in N}\lambda_n n\in radR$ and $h\in H$, the group automorphism $\eta(h):N\rightarrow N$ induces an automorphism $\eta_h$ of $R$, and $h(\sum_{n\in N}\lambda_n n)=\sum_{n\in N}\lambda_n \eta(h)(n)h=\eta_{h}(\sum_{n\in N}\lambda_n n)h$. Since $\eta_{h}(radR)=radR$, $\eta_{h}(\sum_{n\in N}\lambda_n)\in radR$. Therefore $B(radR)\subseteq (radR)B$. Similarly, it can be shown that $(radR)B\subseteq B(radR)$. So the assumption $(a')$ holds.

The $R$-homomorphism $\phi$ is given by $kH\otimes_{k}(kN/radkN)\rightarrow (kN/radkN)\otimes_{kN}kG$, $h\otimes\overline{n}\mapsto \overline{1}\otimes hn$. We have $(kN/radkN)\otimes_{kN}kG\cong (kN/radkN)\otimes_{kN}kN_{G}(N)\oplus(\oplus_{i=1}^{t}(kN/radkN)\otimes_{kN}kNg_i N)$ as $R$-modules, where each $g_i$ belongs to $G-N_{G}(N)$ such that $G-N_{G}(N)$ is a disjoint union of all $Ng_i N$s. Since $|Ng_i N|=|N|^2$, $kNg_i N\cong R\otimes_{k}R$ as $R^e$-modules, so each $(kN/radkN)\otimes_{kN}kNg_i N$ is a projective $R$-module. Moreover, the image of $\phi$ is $(kN/radkN)\otimes_{kN}kN_{G}(N)$. Since $(kN/radkN)\otimes_{kN}kN_{G}(N)\cong\oplus_{h\in H}(kN/radkN)\otimes_{kN}kNh$, $\mathrm{dim}_{k}((kN/radkN)\otimes_{kN}kN_{G}(N))=|H|\mathrm{dim}_{k}(kN/radkN)=\mathrm{dim}_{k}(kH\otimes_{k}(kN/radkN))$, so $\phi$ induces an $R$-isomorphism from $kH\otimes_{k}(kN/radkN)$ to $(kN/radkN)\otimes_{kN}kN_{G}(N)$. Therefore $\phi$ is an isomorphism in $\underline{\mathrm{mod}}$-$R$ and the assumption $(b)$ holds.

Since the trivial $kH$-module $k$ has a periodic free resolution, by Lemma \ref{periodic resolution} the $kH$-$kH$-bimodule $kH$ also has a periodic free resolution. Then the assumption $(c)$ holds. Assume the periodic free resolution of the trivial $kH$-module $k$ is given by the exact sequence $0\rightarrow k\rightarrow F_{n-1}\rightarrow\cdots\rightarrow F_1\rightarrow F_0\rightarrow k\rightarrow 0$, where $F_0$, $\cdots$, $F_{n-1}$ are free $kG$-modules. Then the exact sequence $0\rightarrow \Phi(k)\rightarrow \Phi(F_{n-1})\rightarrow\cdots\rightarrow \Phi(F_1)\rightarrow \Phi(F_0)\rightarrow \Phi(k)\rightarrow 0$ gives a periodic free resolution of the $kH$-$kH$-bimodule $kH$, where $\Phi=-\otimes_{k}kH$ is the functor defined in the proof of Lemma \ref{periodic resolution}.

Let $f:kH\rightarrow kH$, $1\mapsto\sum_{h\in H}\lambda_h h$ be a morphism in $\mathrm{mod}$-$kH$, then $\Phi(f)$ is isomorphic to the $kH$-$kH$-homomorphism $\widetilde{f}:kH\otimes_{k}kH\rightarrow kH\otimes_{k}kH$, $1\otimes 1\mapsto\sum_{h\in H}\lambda_h h^{-1}\otimes h$, by the isomorphism $\Phi(kH)\rightarrow kH\otimes_{k}kH$, $g\otimes h\mapsto hg^{-1}\otimes g$. Since for any $n\in N$, $(\sum_{h\in H}\lambda_h h^{-1}\otimes h)n=\sum_{h\in H}\lambda_h h^{-1}\otimes \eta(h)(n)h=\sum_{h\in H}\lambda_h h^{-1}\eta(h)(n)\otimes h=\sum_{h\in H}\lambda_h \eta(h^{-1})(\eta(h)(n))h^{-1}\otimes h=n(\sum_{h\in H}\lambda_h h^{-1}\otimes h)$ in $kG\otimes_{kN}kG$, there is a $kG$-$kG$-homomorphism $\alpha: kG\otimes_{kN}kG\rightarrow kG\otimes_{kN}kG$ such the diagram
$$\xymatrix{
kH\otimes_{k}kH\ar[r]^{\widetilde{f}}\ar[d] & kH\otimes_{k}kH\ar[d] \\
kG\otimes_{kN}kG\ar[r]^{\alpha} & kG\otimes_{kN}kG \\
 }$$
commutes, where the vertical morphisms are the obvious one. Moreover, for any $kH$-homomorphism $g:kH\rightarrow k$, $1\mapsto \lambda$, $\Phi(f)$ is isomorphic to the $kH$-$kH$-homomorphism $\widetilde{g}:kH\otimes_{k}kH\rightarrow kH$, $1\otimes 1\mapsto\lambda$. Therefore there is a $kG$-$kG$-homomorphism $\beta: kG\otimes_{kN}kG\rightarrow kG$ such the diagram
$$\xymatrix{
kH\otimes_{k}kH\ar[r]^(0.6){\widetilde{g}}\ar[d] & kH\ar@{=}[d] \\
kG\otimes_{kN}kG\ar[r]^(0.6){\beta} & kG \\
 }$$
commutes. Since each $F_i$ is a free $kH$-module, the assumption $(e)$ holds.

Each $kH$-homomorphism $u:k\rightarrow kH$ maps $1$ to some $\lambda(\sum_{h\in H}h)$, where $\lambda\in k$. Then $\Phi(u)$ is isomorphic to the $kH$-$kH$-homomorphism $\widetilde{u}:kH\rightarrow kH\otimes_{k}kH$, $1\mapsto\lambda(\sum_{h\in H}h^{-1}\otimes h)$. Since for every $n\in N$, $(h^{-1}\otimes h)n=h^{-1}\otimes\eta(h)(n)h= h^{-1}\eta(h)(n)\otimes h=\eta(h^{-1})(\eta(h)(n))h^{-1}\otimes h=n(h^{-1}\otimes h)$ in $kG\otimes_{kN}kG$, the image $x$ of $\widetilde{u}(1)$ in $kG\otimes_{kN}kG$ satisfies $rx=xr$ for every $r\in R=kN$. Therefore the assumption $(d)$ holds.
\end{proof}

Suppose the trivial $kH$-module $k$ has a periodic free resolution of periodic $n$, then by Lemma \ref{periodic resolution}, $B=kH$ also has a periodic free resolution of periodic $n$. Let $\rho$ be the stable auto-equivalence of $A=kG$ in Theorem \ref{auto-equivalence 2} with respect to this periodic free resolution of $B$. Similar to Proposition \ref{relative syzygy}, we have following proposition.

\begin{Prop}\label{relative syzygy 2}
For the trivial $kG$-module $k$, $\rho(k)\cong \Omega_{G/N}^{n}(k)$, where $\Omega_{G/N}(M)$ denotes the kernel of some relatively $kN$-projective cover of $M$.
\end{Prop}

\begin{proof}
Consider $B=kH$ as a module over $R=kN$, where each $n\in N$ acts trivially on $B$. Let $\psi:B\rightarrow k\otimes_{R}A$, $h\mapsto 1\otimes h$ be a $k$-linear homomorphism, where $k$ denotes the trivial $R$-module. Since for any $h\in H$ and $n\in N$, $(1\otimes h)n=1\otimes hn=1\otimes \eta(h)(n)h=1\otimes h$ in $k\otimes_{R}A$, $\psi$ is also an $R$-homomorphism. Since $k\otimes_{R}A\cong k\otimes_{kN}kN_{G}(N)\oplus(\oplus_{i=1}^{t}k\otimes_{kN}kNg_i N)$ as $R$-modules, where each $g_i$ belongs to $G-N_{G}(N)$ such that $G-N_{G}(N)$ is a disjoint union of all $Ng_i N$s, $\psi$ is an isomorphism in $\underline{\mathrm{mod}}$-$R$. The rest of the proof is similar to that of Proposition \ref{relative syzygy}.
\end{proof}

\begin{Ex1}
Let $k$ be a field of characteristic $2$ which contains cubic roots of unity, $G=S_4$ be the symmetric group on $4$ letters, and $A=kG$. Let $e_1=1+(123)+(132)$, $e_2=1+\omega(123)+\omega^{2}(132)$, $e_3=1+\omega^{2}(123)+\omega(132)$ be three idempotents of $A$, where $\omega\in k$ is a cubic root of unity. Then $1=e_1+e_2+e_3$ is a decomposition of $1$ into primitive orthogonal idempotents. The basic algebra of $A$ is $\Lambda=fAf$, where $f=e_1+e_2$. It can be shown that $\Lambda$ is given by the quiver
$$\xymatrix{
1\ar@/^/[r]^{\alpha}\ar@(dl,ul)^{\delta} & 2\ar@/^/[l]^{\beta}\ar@(ur,dr)^{\gamma} \\
 }$$
with relations $\alpha\beta=\delta^2=\gamma\alpha=\gamma\beta=0$ and $\alpha\delta\beta=\gamma^2$.

$(i)$ Let $S=\langle (12)\rangle$ be a subgroup of $G$, then $N_{G}(S)=C_{G}(S)=S\times L$, where $L=\langle (34)\rangle$. By Lemma \ref{double coset}, $|SgS|=|S|^2$ for any $g\in G-N_{G}(S)$. Let $R=kS$, $B=kL$. Since the trivial $B$-module $k$ satisfies $\Omega_{B}(k)\cong k$, by Proposition \ref{auto-equivalence given by semidirect product}, the triple $(A,R,B)$ defines a stable auto-equivalence $\rho$ of $A$. Moreover, $\rho$ is induced by the functor $-\otimes_{A}K$, where $K$ is the kernel of the $A^e$-homomorphism $A\otimes_{R}A\rightarrow A$, which is given by multiplication. Since $\Lambda$ is Morita equivalent to $A$, the stable auto-equivalence $\rho$ induces a stable auto-equivalence $\mu$ of $\Lambda$. It can be shown that $\mu(1)=$ \xymatrix@R=0.0pc@C=0.0pc {
	2  \\
	1   \\
}
and $\mu(2)=\Omega_{\Lambda}(2)=$ \xymatrix@R=0.0pc@C=0.0pc {
	1&&& \\
	&1&&2   \\
    &&2& \\
}.

$(ii)$ Let $N=\{(1),(12),(34),(12)(34)\}$ be a subgroup of $G$, then \\ $N_{G}(N)=\{(1),(12),(34),(12)(34),(13)(24),(1324),(14)(23),(1423)\}=N\rtimes H$, where \\ $H=\langle (13)(24)\rangle$. A calculation shows that $G=N_{G}(N)\cup N(13)N$, where $|N(13)N|=|N|^2$. Let $R'=kN$, $B'=kH$. Since the trivial $B'$-module $k$ satisfies $\Omega_{B'}(k)\cong k$, by Proposition \ref{auto-equivalence given by semidirect product}, the triple $(A,R',B')$ defines a stable auto-equivalence $\rho'$ of $A$. Moreover, $\rho'$ is induced by the functor $-\otimes_{A}K'$, where $K'$ is the kernel of the $A^e$-homomorphism $A\otimes_{R'}A\rightarrow A$, which is given by multiplication. Let $\mu'$ be the stable auto-equivalence of $\Lambda$ induced by $\rho'$. It can be shown that $\mu'(1)=$ \xymatrix@R=0.0pc@C=0.0pc {
	&2&  \\
	1&&2  \\
}
and $\mu'(2)=\Omega_{\Lambda}^{-2}(2)=$ \xymatrix@R=0.0pc@C=0.0pc {
	&&1&&2&& \\
	&1&&2&&1&   \\
    2&&&&&&1 \\
}.

$(iii)$ Let $P=\langle (1324)\rangle$ be a subgroup of $G$, then \\ $N_{G}(P)=\{(1),(12),(34),(12)(34),(13)(24),(1324),(14)(23),(1423)\}=P\rtimes Q$, where \\ $Q=\langle (12)\rangle$. We have $G=N_{G}(P)\cup P(13)P$, where $|P(13)P|=|P|^2$. Let $R''=kP$, $B''=kQ$. Similar to case (2) above, the triple $(A,R'',B'')$ defines a stable auto-equivalence $\rho''$ of $A$, which is induced by the functor $-\otimes_{A}K''$, where $K''$ is the kernel of the $A^e$-homomorphism $A\otimes_{R''}A\rightarrow A$. Let $\mu''$ be the stable auto-equivalence of $\Lambda$ induced by $\rho''$, then $\mu''(1)=$ \xymatrix@R=0.0pc@C=0.0pc {
	&2&  \\
	1&&2  \\
}
and $\mu''(2)=\Omega_{\Lambda}^{-2}(2)=$ \xymatrix@R=0.0pc@C=0.0pc {
	&&1&&2&& \\
	&1&&2&&1&   \\
    2&&&&&&1 \\
}, which is same as Case $(ii)$.
\end{Ex1}

\subsection{}

In this subsection, we consider a class of non-local Brauer graph algebras and construct stable auto-equivalences over them. In general, such stable auto-equivalences are not induced by derived auto-equivalences.
\begin{Ex1}
Let $A$ be the Brauer graph algebra given by the Brauer graph
\begin{tikzpicture}
\draw (0,0) circle (0.5);
\draw (1,0) circle (0.5);
\fill (0.5,0) circle (0.5ex);
\node at(0.75,0) {n};
\draw (0.6,-0.15) rectangle (0.9,0.15);
\end{tikzpicture}, where $n\geq 1$.
Then $A$ is given by the quiver
$$\xymatrix{
1\ar@/^/[r]^{\gamma}\ar@(dl,ul)^{\alpha} & 2\ar@/^/[l]^{\delta}\ar@(ur,dr)^{\beta} \\
 }$$
with relations $(\alpha\delta\beta\gamma)^{n}=(\delta\beta\gamma\alpha)^{n}$, $(\beta\gamma\alpha\delta)^{n}=(\gamma\alpha\delta\beta)^{n}$, $\alpha^2=\delta\gamma=\beta^2=\gamma\delta=0$. Let $R=k[\alpha]\times k[\beta]$, $B=k[x]$ be two subalgebras of $A$, where $x=(\delta\beta\gamma\alpha)^{n-1}\delta\beta\gamma+(\gamma\alpha\delta\beta)^{n-1}\gamma\alpha\delta$. The triple $(A,R,B)$ satisfies Assumption 1 in Section 3.

$(1)$ If $char(k)=2$, then $B$ has a periodic free $B^e$-resolution $0\rightarrow B\rightarrow B\otimes_{k}B\xrightarrow{\mu} B\rightarrow 0$ of period $1$, where $\mu$ is the map given by multiplication. According to Theorem \ref{auto-equivalence}, the functor $-\otimes_{A}K$ induces a stable auto-equivalence of $A$, where $K$ is the kernel of the $A^e$-homomorphism $A\otimes_{R}A\rightarrow A$ given by multiplication. Let $S_i$ be the simple $A$-module which corresponds to the vertex $i$. A calculation shows that $S_1\otimes_{A}K\cong rad(e_1 A/\alpha A)$ and $S_2\otimes_{A}K\cong rad(e_2 A/\beta A)$. Note that neither $S_1\otimes_{A}K$ nor $S_2\otimes_{A}K$ belongs to the $\Omega_{A}$-orbit of any simple $A$-module.

When $n=2$, we have $e_1 A=$ \xymatrix@R=0.0pc@C=0.0pc {
	&1& \\
	1&&2 \\
    2&&2 \\
    2&&1 \\
    1&&1 \\
    1&&2 \\
    2&&2 \\
    2&&1 \\
    &1& \\
} and $e_2 A=$ \xymatrix@R=0.0pc@C=0.0pc {
	&2& \\
	2&&1 \\
    1&&1 \\
    1&&2 \\
    2&&2 \\
    2&&1 \\
    1&&1 \\
    1&&2 \\
    &2& \\
}. Let $X=S_1\otimes_{A}K\cong rad(e_1 A/\alpha A)$, then $X$ is the uniserial $A$-module \xymatrix@R=0.0pc@C=0.0pc {
	2 \\
	2 \\
    1 \\
    1 \\
    2 \\
    2 \\
    1 \\
}. Let $\Lambda=\mathrm{End}_{A}(A\oplus S_1)$ and $\Gamma=\mathrm{End}_{A}(A\oplus X)$. By the construction in \cite[Corollary 1.2]{LX2007}, there is a stable equivalence of Morita type between $\Lambda$ and $\Gamma$. The Cartan matrix $C_{\Lambda}$ of $\Lambda$ is given by \[ C_{\Lambda}= \begin{pmatrix}
8 & 8 & 1 \\
8 & 8 & 0 \\
1 & 0 & 1
\end{pmatrix}, \]
and the Cartan matrix $C_{\Gamma}$ of $\Gamma$ is given by \[ C_{\Gamma}= \begin{pmatrix}
8 & 8 & 3 \\
8 & 8 & 4 \\
3 & 4 & 2
\end{pmatrix}. \]
A calculation shows that $C_{\Lambda}$ is congruent to \[ M=\begin{pmatrix}
-1 & 0 & 0 \\
0 & 8 & 0 \\
0 & 0 & 1
\end{pmatrix} \]
 over integers and $C_{\Gamma}$ is congruent to \[ N=\begin{pmatrix}
0 & 0 & -1 \\
0 & 8 & 0 \\
-1 & 0 & 0
\end{pmatrix} \] over integers. If a matrix \[ \begin{pmatrix}
a_{11} & a_{12} & a_{13} \\
a_{21} & a_{22} & a_{23} \\
a_{31} & a_{32} & a_{33}
\end{pmatrix} \] is congruent to $N$
over integers, then it can be shown that $a_{11}$ is even. Therefore the matrices $M$ and $N$
are not congruent over integers. So the matrices $C_{\Lambda}$ and $C_{\Gamma}$ are also not congruent over integers, which implies that $\Lambda$ and $\Gamma$ are not derived equivalent. According to \cite[Proposition 6.1]{HX2010}, the stable auto-equivalence of $A$ induced by the functor $-\otimes_{A}K$ cannot be lifted to a derived auto-equivalence.

$(2)$ If $k$ is a field of arbitrary characteristic, then $B$ has a periodic free $B^e$-resolution $0\rightarrow B\rightarrow B\otimes_{k}B\xrightarrow{f} B\otimes_{k}B\xrightarrow{\mu} B\rightarrow 0$ of period $2$, where $f(1\otimes 1)=1\otimes x-x\otimes 1$ and $\mu$ is the map given by multiplication. According to Theorem \ref{auto-equivalence}, the functor $-\otimes_{A}K'$ induces a stable auto-equivalence of $A$, where $K'$ is given by the short exact sequence $0\rightarrow K'\rightarrow (A\otimes_{R}A)\oplus P\xrightarrow{(h_1,h_2)} K\rightarrow 0$ of $A^e$-modules. Here $K$ is the kernel of the $A^e$-homomorphism $A\otimes_{R}A\rightarrow A$ given by multiplication, $h_1(1\otimes 1)=1\otimes x-x\otimes 1$, and $h_2:P\rightarrow K$ is the projective cover of $K$ as an $A^e$-module. A calculation shows that $S_1\otimes_{A}K'$ (resp. $S_2\otimes_{A}K'$) is isomorphic to the $A$-module $X_1$ (resp. $X_2$) in $\underline{\mathrm{mod}}$-$A$, where $X_1$ (resp. $X_2$) is given by the short exact sequence $0\rightarrow X_1\rightarrow (e_1 A/\alpha A)\oplus e_2 A\xrightarrow{(u_1,u_2)} rad(e_1 A/\alpha A)\rightarrow 0$ (resp. the short exact sequence $0\rightarrow X_2\rightarrow (e_2 A/\beta A)\oplus e_1 A\xrightarrow{(v_1,v_2)} rad(e_2 A/\beta A)\rightarrow 0$), where $u_1(\overline{e_1})=\overline{(\delta\beta\gamma\alpha)^{n-1}\delta\beta\gamma}$ (resp. $v_1(\overline{e_2})=\overline{(\gamma\alpha\delta\beta)^{n-1}\gamma\alpha\delta}$) and $u_2:e_2 A\rightarrow rad(e_1 A/\alpha A)$ (resp. $v_2:e_1 A\rightarrow rad(e_2 A/\beta A)$) is the projective cover of $rad(e_1 A/\alpha A)$ (resp. $rad(e_2 A/\beta A)$). Note that neither $X_1$ nor $X_2$ belongs to the $\Omega_{A}$-orbit of any simple $A$-module.
\end{Ex1}

\end{document}